\documentclass[12pt]{elsarticle}
\usepackage[left=1in,right=1in, top=1.2in,bottom=1.2in]{geometry}
\usepackage{EJGTAart}
\usepackage{times}
\usepackage{amssymb,amsthm,latexsym,amsmath,epsfig,pgf}

\usepackage{mathabx}
\usepackage{algorithm}
\usepackage[noend]{algorithmic}
\usepackage{epstopdf}
\usepackage[normalem]{ulem}

\usepackage{graphicx}
\usepackage{tkz-graph}
\usetikzlibrary{calc,arrows.meta,positioning,fit}
\usetikzlibrary{decorations.markings}

\GraphInit[vstyle = Normal] 
\tikzset{
  LabelStyle/.style = {font = \tiny\bfseries },
  VertexStyle/.append style = { inner sep=5pt,
                                font = \tiny\bfseries},
  EdgeStyle/.append style = {->} }
\usetikzlibrary{arrows, shapes, positioning}

\newcommand{\overunderslash}[2]{{{\underset{#1}{\overset{#2}{/}}}}}

\volume{{\bf -} (-)}

\firstpage{1}

\runauth{
On the decomposition of $n$-partite graphs based on a vertex-removing synchronised graph product
\hspace{2ex} $\arrowvert$\hspace{2ex}
Antoon H. Boode}

\newtheorem{theorem}{Theorem}[section]
\newtheorem*{theorem A}{Theorem A}
\newtheorem*{theorem B}{N\"olker's Theorem}
\newtheorem{lemma}{Lemma}[section]

\theoremstyle{remark}

\theoremstyle{remark}

\newtheorem{claim}{Claim}[section]

\usepackage{csquotes}  
\usepackage{crop}

\makeatletter

\theoremstyle{definition}

\begin{document}
\begin{frontmatter}
% Write your paper title here
\papertitle{On the decomposition of $n$-partite graphs based on a vertex-removing synchronised graph product}

%% use optional labels to link authors explicitly to addresses. The label may be more than one, use comma to separate

%\author[label1,label]{Edy Tri Baskoro${}^1$}

\author[label1]{Antoon H. Boode}

\address[label1]{\small {Robotics Research Group, \\InHolland University of Applied Science\\ The Netherlands}
\vspace*{2.5ex} 
 \normalfont ton.boode@inholland.nl}

\begin{abstract}
\noindent
Recently, we have introduced and modified graph-decomposition theorems based on a graph product motivated by applications in the context of synchronising periodic real-time processes. 
This vertex-removing synchronised product (VRSP) is based on modifications of the well-known Cartesian product and is closely related to the synchronised product due to W\"ohrle and Thomas. 
Here, we introduce a new graph-decomposition theorem based on the VRSP that decomposes an edge-labelled acyclic $n$-partite multigraph where all labels are the same.

\let\thefootnote\relax\footnotetext{Received: May 2022,
\quad Revised: ,\quad   Accepted: .
\\[3ex]
  }
  
\end{abstract}

\begin{keyword}
% Separate keyword by \sep
Vertex Removing Synchronised Graph Product 
\sep Product Graph
\sep Graph Decomposition
\sep Synchronising Processes

% Write the classification number

DOI: %10.5614/ejgta.2015.3.2.5

\end{keyword}

\end{frontmatter}

%%%%%%%%%%%%
\section{Introduction}\label{sec:intro}
Recently, we have introduced three graph-decomposition theorems in 
~\cite{dam},~\cite{ejgta1} and ~\cite{ejgta2} based on a graph product motivated by applications in the context of synchronising periodic real-time processes, in particular in the field of robotics. 
More on the background, definitions and applications can be found in two conference contributions \cite{boode2014cpa, boode2013cpa}, four journal papers \cite{ejgta2, ejgta1, dam,ejgta}, the thesis of the %first
author~\cite{boodethesis} and on ArXiv~\cite{mod-arxiv}. 
We repeat some of the definitions in Section~\ref{sec:term} for convenience.
In Section~\ref{sec:decomp}, we state and prove two lemmas on bipartite and 3-partite graphs which we use to state and prove the decomposition theorem on $n$-partite graphs. 

%The decomposition of graphs is well known in the literature.
%For example, a decomposition can be based on the partition of a graph into edge disjoint subgraphs. 
%In our case, the decomposition is based on the contractions of subsets of the vertices of the graph $G$, in such a manner that if $V'_1,\ldots,V'_m \subseteq V(G)$ are contracted giving a graph $G'$ and $V''_1,\ldots,V''_n \subseteq V(G)$ are contracted giving a graph $G''$, we have that the vertex-removing synchronised product (VRSP) of $G'$ and $G''$ is isomorphic to $G$.

%The rest of the paper is organised as follows.
%In the next two sections, we recall the formal graph definitions (in Section~\ref{sec:term}) and we state and prove (in Section~\ref{sec:decomp}) the bipartite and 3-partite graph decomposition lemmas and the $n$-partite graph decomposition theorem.

\section{Terminology and notation}\label{sec:term}
In order to avoid duplication we refer the interested reader to \cite{dam} or \cite{mod-arxiv} for background, definitions and more details.
Furthermore, we use the textbook of Bondy and Murty~\cite{GraphTheory} for terminology and notation we have not specified here, or in \cite{dam} or in \cite{mod-arxiv}. 
For convenience, we repeat a few definitions that are especially important for the decomposition of an $n$-partite graph.

Let $G$ be an edge-labelled acyclic directed multigraph with a vertex set $V$, an arc set $A$, a set of label pairs $L$ and two mappings.
The first mapping $\mu: A\rightarrow V\times V$ is an incidence function that identifies the {\em tail\/} and {\em head\/} of each arc $a\in A$. 
In particular, $\mu(a)=(u,v)$ means that the arc $a$ is directed from $u\in V$ to $v\in V$, where $tail(a)=u$ and $head(a)=v$. We also call $u$ and $v$ the {\em ends\/} of $a$. 
The second mapping $\lambda :A\rightarrow L$ assigns a label pair $\lambda(a)=(\ell(a),t(a))$ to each arc $a\in A$, where $\ell(a)$ is a string representing the (name of an) action and $t(a)$ is the {\em weight\/} of the arc $a$.

If $X\subseteq V(G)$, then the {\em subgraph of $G$ induced by $X$\/}, is the graph on vertex set $X$ containing all the arcs of $G$ which have both their ends in $X$ (together with $L$, $\mu$ and $\lambda$ restricted to this subset of the arcs).

If $X\subseteq A(G)$, then the {\em subgraph of $G$ arc-induced by $X$\/} is the graph on arc set $X$ containing all the vertices of $G$ which are an end of an arc in $X$ (together with $L$, $\mu$ and $\lambda$ restricted to this subset of the arcs). 

Let $G_i$ and $G_j$ be two disjoint graphs.
An arc $a\in A(G_i)$ with label pair $\lambda(a)$ is a \emph{synchronising arc} with respect to $G_j$, if and only if there exists an arc $b\in A(G_j)$ with label pair $\lambda(b)$ such that $\lambda(a)=\lambda(b)$.
Furthermore, an arc $a$ with label pair $\lambda(a)$ of $G_i\boxtimes G_j$ or $G_i\boxbackslash G_j$ is a \emph{synchronous} arc, whenever there exist a pair of arcs $a_i\in A(G_i)$ and $a_j\in A(G_j)$ with $\lambda(a)=\lambda(a_i)=\lambda(a_j)$.
Analogously, an arc $a$ with label pair $\lambda(a)$ of $G_i\boxtimes G_j$ or $G_i\boxbackslash G_j$ is an \emph{asynchronous} arc, whenever $\lambda(a)\notin L_i$ or $\lambda(a)\notin L_j$.

For disjoint nonempty sets $X,Y\subseteq V(G)$, $[X,Y]$ denotes the set of arcs of $G$ with one end in $X$ and one end in $Y$. If the head of the arc $a\in [X,Y]$ is in $Y$, we call $a$ a {\em forward arc\/} (of $[X,Y]$); otherwise, we call $a$ a {\em backward arc\/}. 

A graph $B$ is called \emph{n-partite} if there exists a partition of nonempty sets $V_1, V_2, \ldots, V_n$ of $V(B)$ into $n$ partite sets (i.e., $V(B) = V_1 \cup \ldots \cup V_n$, $V_i \cap V_j = \emptyset, i\neq j, i,j\in\{1,\ldots,n\}$) such that every arc of $B$ has its head vertex and tail vertex in different partite sets. 
The $n$-partite graph is denoted as $B(V_1,\ldots,V_n)$.
A bipartite graph $B(V_1, V_2)$ is called complete if, for every pair $x \in V_1$, $y \in V_2$, there is an arc $a$ met $\mu(a)=(x,y)$ or $\mu(a)=(y,x)$ in $B(V_1, V_2)$.
%We call $B(V_1, V_2)$ a trivial bipartite graph if $|V_1|=|V_2|=1$.

%A bipartite graph $B(V_1 , V_2)$ is a clean bipartite graph if all subgraphs $B(V'_1 , V'_2)$ of $B(V_1 , V_2)$ are complete, where each subgraph $B(V'_1 , V'_2)$ is arc-induced by all arcs in $[V_1 , V_2]$ with the same label pair and, $[V_1 , V_2]$ has no backward arcs or $[V_1 , V_2]$ has no forward arcs.

%We call a $n$-partite graph $B(V_1, \ldots,V_p)$ a \emph{clean n-partite graph} if all subgraphs $B(V'_i,V'_j)$ of $B(V_1,\ldots,V_p)$ are complete, where the subgraph $B(V'_i,V'_j)$ is arc-induced by all arcs in $[V_i,V_j]$ with the same label pair,and, $[V_i,V_j]$ has no backward arcs or $[V_i,V_j]$ has no forward arcs.

Informally, the vertex-removing synchronised product (VRSP) starts from the well-known Cartesian product, and is based on a reduction of the number of arcs and vertices due to the presence of synchronising arcs, i.e., arcs with the same label. This reduction is done in two steps: in the first step synchronising pairs of arcs from $G_1$ and $G_2$ are replaced by one (diagonal) arc, all other synchronising arcs are removed from the Cartesian product, giving the intermediate product; in the second step, vertices (and the arcs with that vertex as a tail) are removed one by one if they have $level > 0$ in the Cartesian product but $level=0$ in what is left of the intermediate product.

\section{The $n$-partite graph-decomposition theorem.}\label{sec:decomp}
We assume that the graphs we want to decompose are connected; if not, we can apply our decomposition results to the components separately. 
Although the decomposition theorems using the VRSP are dealing with edge-labelled graphs where the labels may be different, in this contribution, we consider only acyclic directed graphs where all labels are the same.

We continue with presenting and proving a decomposition lemma of a bipartite graph, given in Lemma ~\ref{lemma1}.
In this lemma, we are going to decompose a complete bipartite graph $B(X,Y),X=\{u_1,\ldots, u_{_{c_{_1}\cdot c_{_2}}}\}$,
$Y=\{v_1,\ldots, v_{_{c_{_3}\cdot c_{_4}}}\}$ where all arcs have the same label. 
But, different from Lemma~3.1 in~\cite{ejgta1}, we contract both $X$ and $Y$ using disjoint subsets $X'_i$ of $X$, disjoint subsets $X''_i$ of $X$, disjoint subsets $Y'_j$ of $Y$  and disjoint subsets $Y''_j$ of $Y$ such that $B(X,Y)\cong B(X,Y)\overunderslash{i=1}{c_1}X'_i\overunderslash{j=1}{c_3}Y'_j$ $\boxbackslash B(X,Y)\overunderslash{i=1}{c_2}X''_i\overunderslash{j=1}{c_4}Y''_j$, where $\bigcup\limits_{i=1}^{c_1}X'_i=\bigcup\limits_{i=1}^{c_2}X''_i=X$ and $\bigcup\limits_{j=1}^{c_3}Y'_j=\bigcup\limits_{j=1}^{c_4}Y''_j=Y$.
If the cardinality of $X$ is a prime number, hence, $c_1=1$ or $c_2=1$, then, assuming $c_2=1$ and, therefore, $c_1=|X|$, the left part of $B(X,Y)\overunderslash{i=1}{c_1}X'_i\overunderslash{j=1}{c_3}Y'_j$ $\boxbackslash B(X,Y)\overunderslash{i=1}{c_2}X''_i\overunderslash{j=1}{c_4}Y''_j$ is contracted such that each vertex $u_i$ of $X$ is replaced by the vertex $\tilde{x}_i$ and in the right part of $B(X,Y)\overunderslash{i=1}{c_1}X'_i\overunderslash{j=1}{c_3}Y'_j$ $\boxbackslash B(X,Y)\overunderslash{i=1}{c_2}X''_i\overunderslash{j=1}{c_4}Y''_j$, $\bigcup\limits_{i=1}^{c_2}X''_i=\bigcup\limits_{i=1}^{1}X''_i=X$ is contracted giving one vertex $\tilde{x}$.
We have similar contractions for $Y$ if $|Y|$ is a prime number.
Even, if $|X|$ and $|Y|$ are not prime numbers we can set $c_2$ and $c_3$ to one.
This leads to the decomposition $B(X,Y)\overunderslash{i=1}{c_1}X'_i/Y \boxbackslash B(X,Y)/X\overunderslash{j=1}{c_4}Y'_j$ which is equivalent to $B(X,Y)/Y\boxbackslash B(X,Y)/X$.
Therefore, Lemma~\ref{lemma1} is a generalisation of Lemma~3.1 in~\cite{ejgta1}. 
Note that for prime numbers for $|X|$ and $|Y|$ the contraction of $X$ to $\tilde{x}$ and $Y$ to $\tilde{y}$ are on opposite sides of the VRSP of $B(X,Y)\overunderslash{i=1}{c_1\cdot c_2}X'_i/Y \boxbackslash B(X,Y)/X\overunderslash{j=1}{c_3\cdot c_4}Y''_j$.
This is because $B(X,Y)\cong B(X,Y)\overunderslash{i=1}{c_1\cdot c_2}X'_i \overunderslash{j=1}{c_3\cdot c_4}Y''_j$ where $X'_i=\{u_{i}\}, Y'_j=\{v_j\}$, and we do not have a decomposition where the decomposed parts are smaller than $B(X,Y)$.

If the cardinality of $X$ is not a prime number then $X$ is partitioned into $c_1$ subsets $X'_i$ with $|X'_i|=c_2$ and $X$ is partitioned into $c_2$ subsets $X''_i$ with $|X''_i|=c_1$, $Y$ is partitioned into $c_3$ subsets $Y'_j$ with $|Y'_j|=c_4$ and $Y$ is partitioned into $c_4$ subsets $Y''_j$ with $|Y''_j|=c_3$. 
This gives the decomposition $B(X,Y)\cong B(X,Y)\overunderslash{i=1}{c_1}X'_i\overunderslash{j=1}{c_3}Y'_j$ $\boxbackslash B(X,Y)\overunderslash{i=1}{c_2}X''_i\overunderslash{j=1}{c_4}Y''_j$.

In Figure~\ref{BiPartiteExampleDecomposition1}, we give an illustrative example, where $|X|=3$ is a prime number and $|Y|=4=2\cdot 2$ is not.
With respect to $B(X,Y)\cong B(X,Y)\overunderslash{i=1}{c_1}X'_i\overunderslash{j=1}{c_3}Y'_j$ $\boxbackslash B(X,Y)\overunderslash{i=1}{c_2}X''_i\overunderslash{j=1}{c_4}Y''_j$ in Figure~\ref{BiPartiteExampleDecomposition1}, we have  that $X'_1=\{u_{1}\}$, $X'_2=\{u_{2}\}$, $X'_3=\{u_{3}\}$, $Y'_1=\{v_{1},v_{2}\}$, $Y'_2=\{v_{3},v_{4}\}$, $X''_1=\{u_{1},u_{2},u_{3}\}$, $Y''_1=\{v_{1},v_{3}\}$ and $Y''_2=\{v_{2},v_{4}\}$.
In this example, we illustrate how we can decompose a complete bipartite graph $B(X,Y)$ where all the arcs have the same label. 
Because all arcs have the same label, the labels are omitted in all figures of this contribution.

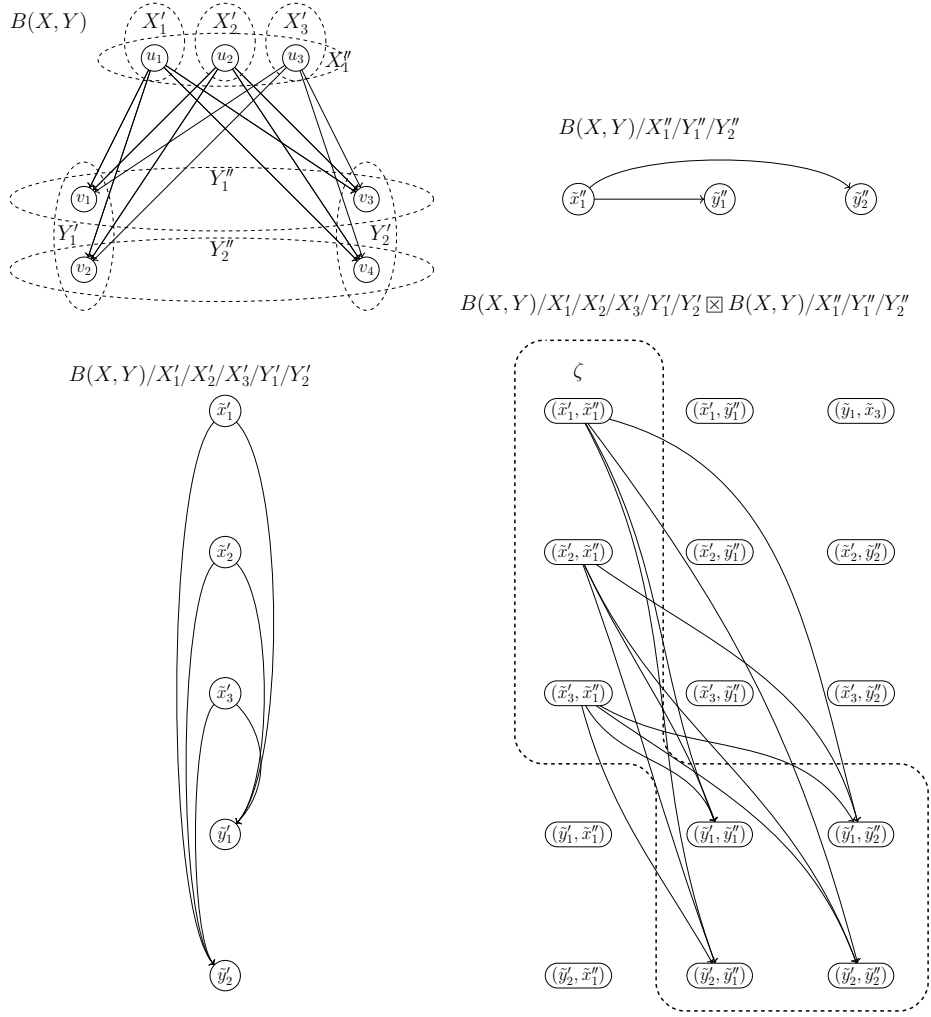
\begin{figure}[H]
\begin{center}
\resizebox{0.75\textwidth}{!}{
\begin{tikzpicture}[font=\sffamily\Large\bfseries]%[scale=0.7]
  \tikzset{VertexStyle/.append style={%,fill=black,
  font=\itshape\large, shape = circle,inner sep = 2pt, outer sep = 0pt,minimum size = 20 pt,draw}}
%  \tikzset{EdgeStyle/.append style={thin}}
  \tikzset{LabelStyle/.append style={font = \itshape}}
  \SetVertexMath
  
  \def\x{0.0}
  \def\y{0.0}

  \node[ellipse,draw,dashed,minimum width = 1.7cm, 
    minimum height = 2.2cm] (e) at (\x-8,\y-1.55) {};  
\node at (\x-8.0,\y-1) {$X'_1$};

  \def\x{2.0}
  \def\y{0.0}

  \node[ellipse,draw,dashed,minimum width = 1.7cm, 
    minimum height = 2.2cm] (e) at (\x-8,\y-1.55) {};  
\node at (\x-8.0,\y-1) {$X'_2$};

  \def\x{4.0}
  \def\y{0.0}

  \node[ellipse,draw,dashed,minimum width = 1.7cm, 
    minimum height = 2.2cm] (e) at (\x-8,\y-1.55) {};  
\node at (\x-8.0,\y-1) {$X'_3$};

  \def\x{1.9}
  \def\y{-4.45}

  \node[ellipse,draw,dashed,minimum width = 12cm, 
    minimum height = 1.8cm] (e) at (\x-8,\y-1.55) {};  
\node at (\x-8.0,\y-1) {$Y''_1$};

  \def\x{1.9}
  \def\y{-6.45}

  \node[ellipse,draw,dashed,minimum width = 12cm, 
    minimum height = 1.8cm] (e) at (\x-8,\y-1.55) {};  
\node at (\x-8.0,\y-1) {$Y''_2$};

  \def\x{1.9}
  \def\y{-0.5}

  \node[ellipse,draw,dashed,minimum width = 7cm, 
    minimum height = 1.5cm] (e) at (\x-8,\y-1.55) {};  
\node at (\x-8.0,\y-1) {};
\node at (\x-4.7,\y-1.6) {$X''_1$};

  \def\x{-2.0}
  \def\y{-5.5}

  \node[ellipse,draw,dashed,minimum width = 1.7cm, 
    minimum height = 4.2cm] (e) at (\x-8,\y-1.55) {};  
\node at (\x-8.45,\y-1.5) {$Y'_1$};
	
  \def\x{6.0}
  \def\y{-5.5}

  \node[ellipse,draw,dashed,minimum width = 1.7cm, 
    minimum height = 4.2cm] (e) at (\x-8,\y-1.55) {};  
\node at (\x-7.6,\y-1.5) {$Y'_2$};

  \def\x{0.0}
  \def\y{2.0}
  
\node at (\x-11,\y-3) {$B(X,Y)$};
\node at (\x+6.0,\y-6) {$B(X,Y)/X''_1/Y''_1/Y''_2$};
\node at (\x-7.0,\y-13) {$B(X,Y)/X'_1/X'_2/X'_3/Y'_1/Y'_2$};
\node at (\x+7,\y-11) {$B(X,Y)/X'_1/X'_2/X'_3/Y'_1/Y'_2\boxtimes B(X,Y)/X''_1/Y''_1/Y''_2$};
\node at (\x+4,\y-12.9) {$\zeta$};
  \def\x{-10.0}
  \def\y{-4.0}
  \Vertex[x=\x+2, y=\y+2.0,L={u_{1}}]{u_1}
  \Vertex[x=\x+4, y=\y+2.0,L={u_{2}}]{u_2}
  \Vertex[x=\x+6, y=\y+2.0,L={u_{3}}]{u_3}
  \Vertex[x=\x+0, y=\y-2,L={v_{1}}]{u_4}
  \Vertex[x=\x+0, y=\y-4,L={v_{2}}]{u_5}
  \Vertex[x=\x+8, y=\y-2,L={v_{3}}]{u_6}
  \Vertex[x=\x+8, y=\y-4,L={v_{4}}]{u_7}

  \Edge(u_1)(u_4) 
  \Edge(u_1)(u_5) 
  \Edge(u_1)(u_6) 
  \Edge(u_1)(u_7) 
  \Edge(u_2)(u_4) 
  \Edge(u_2)(u_5) 
  \Edge(u_2)(u_6) 
  \Edge(u_2)(u_7) 

  \Edge(u_1)(u_4) 
  \Edge(u_1)(u_5) 
  \Edge(u_1)(u_6) 
  \Edge(u_1)(u_7) 
  \Edge(u_2)(u_4) 
  \Edge(u_2)(u_5) 
  \Edge(u_2)(u_6) 
  \Edge(u_2)(u_7) 
  \Edge(u_3)(u_4) 
  \Edge(u_3)(u_5) 
  \Edge(u_3)(u_6) 
  \Edge(u_3)(u_7)

  \def\x{4.0}
  \def\y{-8.0}
  \Vertex[x=\x+0, y=\y+2.0,L={\tilde{x}''_{1}}]{u_1}
  \Vertex[x=\x+4, y=\y+2.0,L={\tilde{y}''_{1}}]{u_2}
  \Vertex[x=\x+8, y=\y+2.0,L={\tilde{y}''_{2}}]{u_3}

  \Edge(u_1)(u_2) 
  \Edge [style={in = 135, out = 45,max distance=1.5cm}](u_1)(u_3) 

   \def\x{-10.0}
  \def\y{-18.0}
  \Vertex[x=\x+4, y=\y+6.0,L={\tilde{x}'_{1}}]{u_1}
  \Vertex[x=\x+4, y=\y+2.0,L={\tilde{x}'_{2}}]{u_2}
  \Vertex[x=\x+4, y=\y-2.0,L={\tilde{x}'_{3}}]{u_3}
  \Vertex[x=\x+4, y=\y-6.0,L={\tilde{y}'_{1}}]{u_4}
  \Vertex[x=\x+4, y=\y-10,L={\tilde{y}'_{2}}]{u_5}

  \Edge[style={-{Stealth[length=3mm, width=2mm]}, in = 45, out = -45,max distance=2cm}](u_1)(u_4) 
  \Edge[-{Stealth[length=3mm, width=2mm]}, style={in = -225, out = 225,max distance=2cm}](u_1)(u_5) 
  \Edge[-{Stealth[length=3mm, width=2mm]},style={in = 45, out = -45,max distance=1.5cm}](u_2)(u_4) 
  \Edge[-{Stealth[length=3mm, width=2mm]},style={in = -225, out = 225,max distance=1.5cm}](u_2)(u_5) 
  \Edge[-{Stealth[length=3mm, width=2mm]}, style={in = 45, out = -45,max distance=1.5cm}](u_3)(u_4) 
  \Edge[-{Stealth[length=3mm, width=2mm]}, style={in = -225, out = 225,max distance=1cm}](u_3)(u_5)

  \tikzset{VertexStyle/.append style={%,fill=black,
  font=\itshape\large, shape = rounded rectangle, inner sep = 2pt, outer sep = 0pt,minimum size = 20 pt,draw}}

\def\x{4}
\def\y{-12}
  
  \Vertex[x=\x+0, y=\y+0,L={(\tilde{x}'_1,\tilde{x}''_1)}]{x_1x_1}
   \Vertex[x=\x+4, y=\y+0,L={(\tilde{x}'_1,\tilde{y}''_1)}]{x_1y_1}
  \Vertex[x=\x+8, y=\y+0,L={(\tilde{y}_1,\tilde{x}_3)}]{y_1x_3}

\def\dy{-4}
  \Vertex[x=\x+0, y=\y+\dy+0,L={(\tilde{x}'_2,\tilde{x}''_1)}]{x_2x_1}
   \Vertex[x=\x+4, y=\y+\dy+0,L={(\tilde{x}'_2,\tilde{y}''_1)}]{x_2y_1}
  \Vertex[x=\x+8, y=\y+\dy+0,L={(\tilde{x}'_2,\tilde{y}''_2)}]{x_2y_2}

\def\dy{-8}
  \Vertex[x=\x+0, y=\y+\dy+0,L={(\tilde{x}'_3,\tilde{x}''_1)}]{x_3x_1}
   \Vertex[x=\x+4, y=\y+\dy+0,L={(\tilde{x}'_3,\tilde{y}''_1)}]{x_3y_1}
  \Vertex[x=\x+8, y=\y+\dy+0,L={(\tilde{x}'_3,\tilde{y}''_2)}]{x_3y_2}

\def\dy{-12}
  \Vertex[x=\x+0, y=\y+\dy+0,L={(\tilde{y}'_1,\tilde{x}''_1)}]{y_1x_1}
   \Vertex[x=\x+4, y=\y+\dy+0,L={(\tilde{y}'_1,\tilde{y}''_1)}]{y_1y_1}
  \Vertex[x=\x+8, y=\y+\dy+0,L={(\tilde{y}'_1,\tilde{y}''_2)}]{y_1y_2}

\def\dy{-16}
  \Vertex[x=\x+0, y=\y+\dy+0,L={(\tilde{y}'_2,\tilde{x}''_1)}]{y_2x_1}
   \Vertex[x=\x+4, y=\y+\dy+0,L={(\tilde{y}'_2,\tilde{y}''_1)}]{y_2y_1}
  \Vertex[x=\x+8, y=\y+\dy+0,L={(\tilde{y}'_2,\tilde{y}''_2)}]{y_2y_2}

 \Edge[labelstyle={xshift=-8pt, yshift=12pt}, style={in = 112, out = -60,min distance=2cm}](x_1x_1)(y_1y_1) 
  \Edge[labelstyle={xshift=-8pt, yshift=12pt}, style={in = 110, out = -65,min distance=2cm}](x_2x_1)(y_1y_1) 
  \Edge[labelstyle={xshift=-8pt, yshift=12pt}, style={in = 110, out = -60,min distance=2cm}](x_3x_1)(y_1y_1) 
 
 \Edge[labelstyle={xshift=-8pt, yshift=12pt}, style={in = 110, out = -62,min distance=2cm}](x_1x_1)(y_2y_1) 
  \Edge[labelstyle={xshift=-8pt, yshift=12pt}, style={in = 110, out = -70,min distance=2cm}](x_2x_1)(y_2y_1) 
  \Edge[labelstyle={xshift=-8pt, yshift=12pt}, style={in = 120, out = -75,min distance=2cm}](x_3x_1)(y_2y_1) 
 
 \Edge[labelstyle={xshift=-8pt, yshift=12pt}, style={in = 105, out = -15,min distance=2cm}](x_1x_1)(y_1y_2) 
  \Edge[labelstyle={xshift=-8pt, yshift=12pt}, style={in = 105, out = -35,min distance=2cm}](x_2x_1)(y_1y_2) 
  \Edge[labelstyle={xshift=-8pt, yshift=12pt}, style={in = 120, out = -30,min distance=2cm}](x_3x_1)(y_1y_2) 
 
 \Edge[labelstyle={xshift=-8pt, yshift=12pt}, style={in = 105, out = -53,min distance=2cm}](x_1x_1)(y_2y_2) 
  \Edge[labelstyle={xshift=-8pt, yshift=/X'_112pt}, style={in = 110, out = -65,min distance=2cm}](x_2x_1)(y_2y_2) 
  \Edge[labelstyle={xshift=-8pt, yshift=12pt}, style={in = 110, out = -35,min distance=2cm}](x_3x_1)(y_2y_2) 

  \def\x{10}
  \def\y{-4.0}

  \def\x{-0.5+3.5}
  \def\y{-12.0}
\draw[circle, -,dashed, very thick,rounded corners=8pt] (\x+0.2,\y+2)--(\x+2.9,\y+2)--(\x+3.4,\y+1.5)--(\x+3.4,\y-9.5)--(\x+3.9,\y-10)--(\x+10.4,\y-10) --(\x+10.9,\y-10.5) -- (\x+10.9,\y-16.5)-- (\x+10.4,\y-17) -- (\x-0.3+4,\y-17) -- (\x-0.8+4,\y-16.5) -- (\x-0.8+4,\y-9.5-1) -- (\x-0.8+3.5,\y-9-1) -- (\x-0.8+0.5,\y-9-1) -- (\x-0.8,\y-8.5-1) -- (\x-0.8,\y-4.5)--(\x-0.8,\y+1.5) -- (\x-0.3,\y+2)--(\x+0.1,\y+2);
\end{tikzpicture}
}
\end{center}
\caption{Decomposition of $B(X,Y)\cong B(X,Y)/X'_1/X'_2/X'_3/Y'_1/Y'_2\boxbackslash B(X,Y)/X''_1/Y''_1/Y''_2$. The set $\zeta$ from the proof of Lemma~\ref{lemma1} and the graph isomorphic to $B(X,Y)$ induced by $\zeta$ in $B(X,Y)/X'_1/X'_2/X'_3/Y'_1/Y'_2\boxtimes B(X,Y)/X''_1/Y''_1/Y''_2$ is indicated within the dotted region (although not all arcs fit into this region).}
  \label{BiPartiteExampleDecomposition1}
\end{figure}
Note that the proof of Lemma~\ref{lemma1} is modelled along the same lines as the proofs of the theorems presented in \cite{ejgta1} and in~\cite{dam}.
%
%=====================================================
%\newpage
\begin{lemma}\label{lemma1}
Let $B(X,Y)$ be a weakly connected complete bipartite graph where the labels of all arcs are the same.
Let $[X,Y]$ contain only forward arcs, or let $[X,Y]$ contain only backward arcs.
Let $|X|=c_1\cdot c_2,|Y|=c_3\cdot c_4, c_1,\ldots,c_4\in~\hspace{-5pt}\mathbb{N}$.
Then there exist $X'_g,X''_h, Y'_i$ and $Y''_j$ such that $B(X,Y)\cong B(X,Y)\overunderslash{g=1}{c_1}X'_g\overunderslash{i=1}{c_3}Y'_i\boxbackslash B(X,Y)\overunderslash{h=1}{c_2}X''_h\overunderslash{j=1}{c_4}Y''_j$.
\end{lemma}
\begin{proof}
Let $[X,Y]$ contain only forward arcs.
%Let $g\in\{1,\ldots,c_1\},h\in\{1,\ldots,c_2\},i\in\{1,\ldots,c_3\}$, $j\in\{1,\ldots,c_4\}$.
It is sufficient to define a mapping $\phi:V(B(X,Y,Z))$ $\rightarrow V(B(X,Y)\overunderslash{g=1}{c_1}X'_g\overunderslash{i=1}{c_3}Y'_i\,\boxbackslash B(X,Y)\overunderslash{h=1}{c_2}X''_h\overunderslash{j=1}{c_4}Y''_j)$ and to prove that $\phi$ is an isomorphism from $B(X,Y)$ to $B(X,Y)\overunderslash{g=1}{c_1}X'_g\overunderslash{i=1}{c_3}Y'_i\boxbackslash B(X,Y,Z)\overunderslash{h=1}{c_2}X''_h$ $\overunderslash{j=1}{c_4}Y''_j$.

Because $B(X,Y)$ is complete there are arcs from each vertex of $X$ to all vertices of $Y$.
Then, without loss of generality, we define $X=\{u_{1,1},\ldots,u_{1,c_{_2}},\ldots,u_{c_{_1},1},\ldots,u_{c_{_1},c_{_2}}\}$.
Now, we can contract $X$ using the sets $X'_1,\ldots, X'_{c_{_1}}$, $X'_g=\{u_{g,1},\ldots,u_{g,c_{_2}}\}$, $|X'_g|=c_2,g=1,\ldots,c_1$. The vertices in the sets $X'_1,\ldots, X'_{c_{_1}}$ are then replaced by the vertices $\tilde{x}'_1,\ldots, \tilde{x}'_{c_{_1}}$, respectively, (note that there are no arcs that have both their ends in $X'_g$), and we can contract $X$ using the sets $X''_{1},\ldots,X''_{c_{_2}},$ $X''_h=\{u_{1,h},\ldots,u_{c_{_1},h}\}$, $|X''_h|=c_1,h=1,\ldots,c_2$. The vertices in the sets $X''_1,\ldots, X''_{c_{_2}}$ are then replaced by the vertices $\tilde{x}''_1,\ldots, \tilde{x}''_{c_{_2}}$, respectively, (note that there are no arcs that have both their ends in  $X''_h$).
Likewise, for $|Y|=c_3\cdot c_4$,  we define $Y=\{v_{1,1},\ldots,v_{1,c_{_4}},\ldots,v_{c_{_3},1},\ldots,v_{c_{_3},c_{_4}}\}$.
Then, we can contract $Y$ using the sets $Y'_1,\ldots, Y'_{c_{_3}}$, $Y'_i=\{v_{i,1},\ldots,v_{i,c_{_4}}\}$, $|Y'_i|=c_4,i=1,\ldots,c_3$. 
The vertices in the sets $Y'_1,\ldots, Y'_{c_{_3}}$ are then replaced by the vertices $\tilde{y}'_1,\ldots, \tilde{y}'_{c_{_3}}$, respectively, (note that there are no arcs that have both their ends in $Y'_k$), and we can contract $Y$ using the sets $Y''_{1},\ldots,Y''_{c_{_4}},Y''_j=\{v_{1,l},\ldots,v_{c_{_3},j}\}, |Y'_j|=c_3,j=1,\ldots,c_4$. The vertices in the sets $Y''_1,\ldots, Y''_{c_{_4}}$ are then replaced by the vertices $\tilde{y}''_1,\ldots, \tilde{y}''_{c_{_4}}$, respectively, (note that there are no arcs that have both their ends in  $Y''_j$). 
 
Consider the mapping $\phi:V(B(X,Y))\rightarrow V(B(X,Y)\overunderslash{g=1}{c_1}X'_g\overunderslash{i=1}{c_3}Y'_i\boxbackslash B(X,Y)\overunderslash{h=1}{c_2}X''_h\overunderslash{j=1}{c_4}Y''_j)$ defined by $\phi(u_{g,h})=(\tilde{x}'_g,\tilde{x}''_h),\phi(v_{i,j})=(\tilde{y}'_i,\tilde{y}''_j)$.
Then $\phi$ is obviously a bijective map if $V(B(X,Y)$ $\overunderslash{g=1}{c_1}X'_g\overunderslash{i=1}{c_3}Y'_i\boxbackslash B(X,Y)\overunderslash{h=1}{c_2}X''_h\overunderslash{j=1}{c_4}Y''_j)=\zeta$, where $\zeta$ is defined as $\zeta=\{(\tilde{x}'_g,\tilde{x}''_h)\mid u_{g,h}\in X,\phi(u_{g,h})=(\tilde{x}'_g,\tilde{x}''_h)\}\cup \{(\tilde{y}'_i,\tilde{y}''_j)\mid v_{i,j}\in Y,\phi(v_{i,j})=(\tilde{y}'_i,\tilde{y}''_j)\}$.
We are going to show this later by arguing all the other vertices (and their labelled arcs) $(\tilde{x}'_g,\tilde{y}''_j), (\tilde{x}'_g,\tilde{z}''_l),(\tilde{y}'_i,\tilde{x}''_h)$ and $(\tilde{y}'_i,\tilde{z}''_l)$ of $B(X,Y)\overunderslash{g=1}{c_1}X'_g\overunderslash{i=1}{c_3}Y'_i$ $\boxtimes B(X,Y)\overunderslash{h=1}{c_2}X''_h\overunderslash{j=1}{c_4}Y''_j$ will disappear from $B(X,Y)\overunderslash{g=1}{c_1}X'_g\overunderslash{i=1}{c_3}Y'_i\boxtimes B(X,Y)\overunderslash{h=1}{c_2}X''_h\overunderslash{j=1}{c_4}Y''_j$. 	
But first we are going to prove the following claim. 
\begin{claim}\label{claim1}
The subgraph of $B(X,Y)\overunderslash{g=1}{c_1}X'_g\overunderslash{i=1}{c_3}Y'_i\boxtimes B(X,Y,Z)\overunderslash{h=1}{c_2}X''_h\overunderslash{j=1}{c_4}Y''_j$ induced by $\zeta$ is isomorphic to $B(X,Y)$.
\end{claim}
\begin{proof}
Firstly, $\phi$ is a bijection from $V(B(X,Y))$ to $\zeta$.
Secondly, an arc $u_{i_{_1},i_{_2}}v_{j_{_1},j_{_2}}$ in $B(X,Y)$ corresponds to the arc $\tilde{x}'_{i_{_1}}\tilde{y}'_{j_{_1}}$ in $B(X,Y)\overunderslash{i=1}{c_1}X'_{i}\overunderslash{j=1}{c_3}Y'_{j}$ and to the arc $\tilde{x}''_{i_{_2}}\tilde{y}''_{j_{_2}}$ in $B(X,Y)\overunderslash{i=1}{c_2}X''_{i}\overunderslash{j=1}{c_4}Y''_{j}$.
Because all arcs of $B(X,Y)\overunderslash{i=1}{c_1}X'_{i}\overunderslash{j=1}{c_3}Y'_{j}$ and $B(X,Y)$ $\overunderslash{i=1}{c_2}X''_{i}\overunderslash{j=1}{c_4}Y''_{j}$ are synchronising arcs, for each pair of arcs $\tilde{x}'_{i_{_1}}\tilde{y}'_{j_{_1}},\tilde{x}''_{i_{_2}}\tilde{y}''_{j_{_2}}$ of $B(X,Y)\overunderslash{i=1}{c_1}X'_{i}\overunderslash{j=1}{c_3}Y'_{j}$ and $B(X,Y)$ $\overunderslash{i=1}{c_2}X''_{i}\overunderslash{j=1}{c_4}Y''_{j}$, respectively, there is an arc $(\tilde{x}'_{i_{_1}},\tilde{x}''_{i_{_2}})(\tilde{y'_{j_{_1}}} \tilde{y}''_{j_{_2}})$ of $B(X,Y)\overunderslash{i=1}{c_1}X'_{i}$ $\overunderslash{j=1}{c_3}Y'_{j}\boxtimes B(X,Y)\overunderslash{i=1}{c_2}X''_{i}\overunderslash{j=1}{c_4}Y''_{j}$ and, therefore an arc $u_{i_{_1},i_{_2}}v_{j_{_1},j_{_2}}$ in $B(X,Y)$ corresponds to an arc $(\tilde{x}'_{i_{_1}},\tilde{x}''_{i_{_2}})(\tilde{y'}_{i_{_1}} \tilde{y}''_{i_{_2}})$ in $B(X,Y)\overunderslash{i=1}{c_1}X'_{i}\overunderslash{j=1}{c_3}Y'_{j}\boxtimes B(X,Y)\overunderslash{i=1}{c_2}X''_{i}\overunderslash{j=1}{c_4}Y''_{j}$.

Hence, the map $\phi$ is a bijjection from $B(X,Y)$ to the subgraph of $B(X,Y)\overunderslash{i=1}{c_1}X'_{i}\overunderslash{j=1}{c_3}Y'_{j}\boxtimes B(X,Y)\overunderslash{i=1}{c_2}X''_{i}\overunderslash{j=1}{c_4}Y''_{j}$ induced by $\zeta$ preserving the arcs and their labels and, therefore, $B(X,Y)$ is isomorphic to the subgraph of $B(X,Y)\overunderslash{i=1}{c_1}X'_{i}\overunderslash{j=1}{c_3}Y'_{j}\boxtimes B(X,Y)\overunderslash{i=1}{c_2}X''_{i}\overunderslash{j=1}{c_4}Y''_{j}$ induced by $\zeta$.
This completes the proof of Claim~\ref{claim1}.
\end{proof}
We continue with the proof of Lemma~\ref{lemma1}. 
It remains to show that all vertices of $V(B(X,Y)$ $\overunderslash{i=1}{c_1}X'_{i}\overunderslash{j=1}{c_3}Y'_{j}\Box B(X,Y)\overunderslash{i=1}{c_2}X''_{i}\overunderslash{j=1}{c_4}Y''_{j})\backslash \zeta$ (and the arcs of which these vertices are an end) disappear from $B(X,Y)\overunderslash{i=1}{c_1}X'_{i}\overunderslash{j=1}{c_3}Y'_{j}\boxtimes B(X,Y)\overunderslash{i=1}{c_2}X''_{i}\overunderslash{j=1}{c_4}Y''_{j}$.
This follows directly by the observation that only the vertices $(\tilde{x}'_i,\tilde{x}''_j)$ have level~0 by definition of the Cartesian product. 
Then, all other vertices $(\tilde{x}'_i,\tilde{y}''_j)$ and  $(\tilde{y}'_i,\tilde{x}''_j)$ have level$>$0 in $B(X,Y)\overunderslash{i=1}{c_1}X'_{i}\overunderslash{j=1}{c_2}Y'_{j}\Box$ $B(X,Y)\overunderslash{i=1}{c_2}X''_{i}\overunderslash{j=1}{c_4}Y''_{j}$.
But these vertices $(\tilde{x}'_i,\tilde{y}''_j)$ and $(\tilde{y}'_i,\tilde{x}''_j)$ have level~0 in $B(X,Y)$ $\overunderslash{i=1}{c_1}X'_{i}\overunderslash{j=1}{c_3}Y'_{j}\boxtimes B(X,Y)\overunderslash{i=1}{c_2}X''_{i}\overunderslash{j=1}{c_4}Y''_{j}$ and are, therefore, removed from $B(X,Y)\overunderslash{i=1}{c_1}X'_{i}\overunderslash{j=1}{c_3}Y'_{j}\boxtimes B(X,Y)\overunderslash{i=1}{c_2}X''_{i}\overunderslash{j=1}{c_4}Y''_{j}$, together with the arcs of which these vertices are an end.
This is because there are no arcs $a$ with $ head(a)=\tilde{x}'_i$ in $B(X,Y)\overunderslash{i=1}{c_1}X'_i$ $\overunderslash{j=1}{c_3}Y'_j$,\ and, therefore, there are no arcs $b$ with $head(b)=(\tilde{x}'_i,\tilde{y}''_j)$ in $B(X,Y)$ $\overunderslash{i=1}{c_1}X'_i$ $\overunderslash{j=1}{c_3}Y'_j\boxtimes B(X,Y)\overunderslash{i=1}{c_2}X''_i\overunderslash{j=1}{c_4}Y''_j$, and there are no arcs $a$ with $head(a)= \tilde{x}''_i$ in $B(X,Y)\overunderslash{i=1}{c_2}X''_i$ $\overunderslash{j=1}{c_4}Y''_j$ and, therefore, there are no arcs $b$ with  $head(b)=(\tilde{y}'_j,\tilde{x}''_i)$ in $B(X,Y)\overunderslash{i=1}{c_1}X'_i$ $\overunderslash{j=1}{c_3}Y'_j\boxtimes B(X,Y)\overunderslash{i=1}{c_2}X''_i\overunderslash{j=1}{c_4}Y''_j$.
Hence, $(\tilde{x}'_i,\tilde{y}''_j)$ and $(\tilde{y}'_j,\tilde{x}''_i)$ must have level~0 in $B(X,Y)\overunderslash{i=1}{c_1}X'_i$ $\overunderslash{j=1}{c_3}Y'_i\boxtimes B(X,Y)\overunderslash{i=1}{c_2}X''_i\overunderslash{j=1}{c_4}Y''_j$.
Then $B(X,Y)\cong B(X,Y)\overunderslash{i=1}{c_1}X'_i\overunderslash{j=1}{c_3}Y'_j\boxbackslash B(X,Y)\overunderslash{i=1}{c_2}X''_i\overunderslash{j=1}{c_4}Y''_j$.
The proof for $[X,Y]$ containing only backward arcs is similar.
This completes the proof of Lemma~\ref{lemma1}.
\end{proof}

We continue with the decomposition of a $3$-partite graph $B(X,Y,Z)$.
In Figure~\ref{BiPartiteExampleDecomposition2}, we show an example of a $3$-partite graph $B(X,Y,Z)$ that is decomposed in graphs $B(X,Y,Z)$ $/X'_1/Y'_1/Y'_2/Y'_3/$ $Z'_1$ and $B(X,Y,Z)/X''_1/Y''_1/Z''_1/Z''_2$. 
Note that the bipartite subgraph induced by the vertex set $Y\cup Z$ is not complete, but the bipartite subgraph arc-induced by the arcs of $[Y, Z]$ is complete.

\begin{figure}[H]
\begin{center}
\resizebox{0.75\textwidth}{!}{
\begin{tikzpicture}[->,>=Latex,shorten >=0pt,auto,node distance=2.5cm,
  main node/.style={circle,fill=blue!10,draw, font=\sffamily\Large\bfseries}]%[scale=0.7]
  \tikzset{VertexStyle/.append style={%,fill=black,
  font=\itshape\large, shape = circle,inner sep = 2pt, outer sep = 0pt,minimum size = 32 pt,draw}}
%  \tikzset{EdgeStyle/.append style={thin}}
  \tikzset{LabelStyle/.append style={font = \itshape}}
  \SetVertexMath
  \def\x{0.0}
  \def\y{2.0}
\node at (\x-8.5,\y-0) {$B(X,Y,Z)$};
\node at (\x-10+6.5,\y-0) {$X=X'_1=X''_1$};
\node at (\x-10+5.5,\y-3) {$Y=Y''_1$};
\node at (\x-10+0.65,\y-4) {$Y'_1$};
\node at (\x-10+3.65,\y-4) {$Y'_2$};
\node at (\x-10+6.65,\y-4) {$Y'_3$};
\node at (\x-10+5.5,\y-6) {$Z=Z'_1$};
\node at (\x-10+0.65,\y-7) {$Z''_1$};
\node at (\x-10+3.65,\y-7) {$Z''_2$};

\node at (\x+4.0,\y-0) {$B(X,Y,Z)/X''/Y''/Z''_1/Z''_2$};
\node at (\x-7.0,\y-8) {$B(X,Y,Z)/X'/Y'_1/Y'_2/Y'_3/Z'$};
\node at (\x+5,\y-8) {$B(X,Y,Z)/X'/Y'_1/Y'_2/Y'_3/Z'\boxtimes B(X,Y,Z)/X''/Y''/Z''_1/Z''_2$};
\node at (\x+3,\y-9.9) {$\zeta$};
  \def\x{-10.0}
  \def\y{-2.0}
  \Vertex[x=\x+4, y=\y+4.0,L={u_{1}}]{u_1}
  \Vertex[x=\x+1, y=\y+1.0,L={u_{2}}]{u_2}
  \Vertex[x=\x+4, y=\y+1.0,L={u_{3}}]{u_3}
  \Vertex[x=\x+7, y=\y+1.0,L={u_{4}}]{u_4}
  \Vertex[x=\x+1, y=\y-2,L={u_{5}}]{u_5}
  \Vertex[x=\x+4, y=\y-2,L={u_{6}}]{u_6}

  \Edge(u_1)(u_2) 
  \Edge(u_1)(u_3) 
  \Edge(u_1)(u_4) 
  \Edge(u_2)(u_5) 
  \Edge(u_2)(u_6) 
  \Edge(u_3)(u_5) 
  \Edge(u_3)(u_6) 

   \def\x{-1.0}
  \def\y{-2.0}
  \Vertex[x=\x+3, y=\y+0.0,L={\tilde{x}''_1}]{u_1}
  \Vertex[x=\x+6, y=\y+0.0,L={\tilde{y}''_1}]{u_2}
  \Vertex[x=\x+9, y=\y+0.0,L={\tilde{z}''_1}]{u_3}
  \Vertex[x=\x+12, y=\y+0.0,L={\tilde{z}''_2}]{u_4}

  \Edge(u_1)(u_2) 
  \Edge(u_2)(u_3) 
  \Edge[style={in = 135, out = 45,min distance=2cm}](u_2)(u_4)

   \def\x{-7.0}
  \def\y{-7.0}
  \Vertex[x=\x+0, y=\y-3.0,L={\tilde{x}'_1}]{u_1}
  \Vertex[x=\x+0, y=\y+0.0,L={\tilde{y}'_3}]{u_2}
  \Vertex[x=\x+0, y=\y-6.0,L={\tilde{y}'_1}]{u_3}
  \Vertex[x=\x+0, y=\y-9.0,L={\tilde{y}'_2}]{u_4}
  \Vertex[x=\x+0, y=\y-12.0,L={\tilde{z}'_1}]{u_5}

  \Edge(u_1)(u_2) 
  \Edge(u_1)(u_3) 
  \Edge[style={in = 45, out = -45,min distance=2cm}](u_1)(u_4) 
  \Edge[style={in = 135, out = -135,min distance=2cm}](u_3)(u_5) 
  \Edge(u_4)(u_5)

  \tikzset{VertexStyle/.append style={%,fill=black,
  font=\itshape\large, shape = rounded rectangle, inner sep = 2pt, outer sep = 0pt,minimum size = 20 pt,draw}}

  \def\x{2.0}
  \def\y{-9.0}
  \Vertex[x=\x+0, y=\y+2.0,L={(\tilde{y}'_3,\tilde{x}''_1)}]{x2x1}
  \Vertex[x=\x+3, y=\y+2.0,L={(\tilde{y}'_3,\tilde{y}''_1)}]{x2x2}
  \Vertex[x=\x+6, y=\y+2.0,L={(\tilde{y}'_3,\tilde{z}''_1)}]{x2x3}
  \Vertex[x=\x+9, y=\y+2.0,L={(\tilde{y}'_3,\tilde{z}''_2)}]{x2x4}

  \def\x{2.0}
  \def\y{-12.0}
  \Vertex[x=\x+0, y=\y+2.0,L={(\tilde{x}'_1,\tilde{x}''_1)}]{x1x1}
  \Vertex[x=\x+3, y=\y+2.0,L={(\tilde{x}'_1,\tilde{y}''_1)}]{x1x2}
  \Vertex[x=\x+6, y=\y+2.0,L={(\tilde{x}'_1,\tilde{z}''_1)}]{x1x3}
  \Vertex[x=\x+9, y=\y+2.0,L={(\tilde{x}'_1,\tilde{z}''_2)}]{x1x4}

  \def\x{2.0}
  \def\y{-15.0}
  \Vertex[x=\x+0, y=\y+2.0,L={(\tilde{y}'_1,\tilde{x}''_1)}]{x3x1}
  \Vertex[x=\x+3, y=\y+2.0,L={(\tilde{y}'_1,\tilde{y}''_1)}]{x3x2}
  \Vertex[x=\x+6, y=\y+2.0,L={(\tilde{y}'_1,\tilde{z}''_1)}]{x3x3}
  \Vertex[x=\x+9, y=\y+2.0,L={(\tilde{y}'_1,\tilde{z}''_2)}]{x3x4}

  \def\x{2.0}
  \def\y{-18.0}
  \Vertex[x=\x+0, y=\y+2.0,L={(\tilde{y}'_2,\tilde{x}''_1)}]{x4x1}
  \Vertex[x=\x+3, y=\y+2.0,L={(\tilde{y}'_2,\tilde{y}''_1)}]{x4x2}
  \Vertex[x=\x+6, y=\y+2.0,L={(\tilde{y}'_2,\tilde{z}''_1)}]{x4x3}
  \Vertex[x=\x+9, y=\y+2.0,L={(\tilde{y}'_2,\tilde{z}''_2)}]{x4x4}

  \def\x{2.0}
  \def\y{-21.0}
  \Vertex[x=\x+0, y=\y+2.0,L={(\tilde{z}'_1,\tilde{x}''_1)}]{x5x1}
  \Vertex[x=\x+3, y=\y+2.0,L={(\tilde{z}'_1,\tilde{y}''_1)}]{x5x2}
  \Vertex[x=\x+6, y=\y+2.0,L={(\tilde{z}'_1,\tilde{z}''_1)}]{x5x3}
  \Vertex[x=\x+9, y=\y+2.0,L={(\tilde{z}'_1,\tilde{z}''_2)}]{x5x4}

  \Edge(x1x1)(x2x2) 
  \Edge(x1x2)(x2x3) 
  \Edge(x1x2)(x2x4) 

  \Edge(x1x1)(x3x2) 
  \Edge(x1x2)(x3x3) 
  \Edge(x1x2)(x3x4) 

  \Edge(x1x1)(x4x2) 
  \Edge(x1x2)(x4x3) 
  \Edge[style={in = 105, out = -30,min distance=2cm}](x1x2)(x4x4) 

  \Edge(x3x1)(x5x2) 
  \Edge[style={in = 120, out = -70,min distance=2cm}](x3x2)(x5x3) 
  \Edge[style={in = 160, out = -65,min distance=2cm}](x3x2)(x5x4) 

  \Edge(x4x1)(x5x2) 
  \Edge(x4x2)(x5x3) 
  \Edge[style={in = 165, out = -30,min distance=2cm}](x4x2)(x5x4)

  \def\x{0.75}
  \def\y{-11.25}
\draw[circle, -,dashed, very thick,rounded corners=8pt] (\x+0.1,\y+1.0)--(\x+0.1,\y+1.5)--(\x+3.1,\y+4.7)--(\x+3.6,\y+4.9)-- (\x+5.4,\y+4.8)-- (\x+5.6,\y+4.5)-- (\x+5.6,\y+4.0)-- (\x+2.5,\y+1.2)-- (\x+5.5,\y-1.5)-- (\x+6.1,\y-4.6)--(\x+6.35,\y-5.25)--(\x+11.5,\y-7.3)--(\x+11.7,\y-7.9)-- (\x+11.5,\y-8.5)-- (\x+6.5,\y-8.5)-- (\x+3.25,\y-5.5) --(\x+2.5,\y-1.5) -- (\x+0.1,\y+0.5)--(\x+0.1,\y+1.0);
  \def\x{-6.8}
  \def\y{+0.8}
  \draw[circle, -,dotted, very thick,rounded corners=8pt] (\x+0.1,\y+1.6)--(\x+0.1,\y+1.8)-- (\x+1.5,\y+1.8)-- (\x+1.5,\y+0.6)-- (\x+0.1,\y+0.6)--(\x+0.1,\y+1.6);
  \def\x{-6.8-3}
  \def\y{+0.8-3}
  \draw[circle, -,dotted, very thick,rounded corners=8pt] (\x+0.1,\y+1.6)--(\x+0.1,\y+1.8)-- (\x+1.5+6,\y+1.8)-- (\x+1.5+6,\y+0.6)-- (\x+0.1,\y+0.6)--(\x+0.1,\y+1.6);
  \def\x{-6.8-3}
  \def\y{+0.8-6}
  \draw[circle, -,dotted, very thick,rounded corners=8pt] (\x+0.1,\y+1.6)--(\x+0.1,\y+1.8)-- (\x+1.5+3,\y+1.8)-- (\x+1.5+3,\y+0.6)-- (\x+0.1,\y+0.6)--(\x+0.1,\y+1.6);
  
    \def\x{-6.8-3}
  \def\y{+0.8-3}
  \draw[circle, -,dotted, very thick,rounded corners=8pt] (\x+0.1,\y+1.6)--(\x+0.1,\y+1.8)-- (\x+1.5,\y+1.8)-- (\x+1.5,\y-0.2)-- (\x+0.1,\y-0.2)--(\x+0.1,\y+1.6);
    \def\x{-6.8}
  \def\y{+0.8-3}
  \draw[circle, -,dotted, very thick,rounded corners=8pt] (\x+0.1,\y+1.6)--(\x+0.1,\y+1.8)-- (\x+1.5,\y+1.8)-- (\x+1.5,\y-0.2)-- (\x+0.1,\y-0.2)--(\x+0.1,\y+1.6);
    \def\x{-6.8+3}
  \def\y{+0.8-3}
  \draw[circle, -,dotted, very thick,rounded corners=8pt] (\x+0.1,\y+1.6)--(\x+0.1,\y+1.8)-- (\x+1.5,\y+1.8)-- (\x+1.5,\y-0.2)-- (\x+0.1,\y-0.2)--(\x+0.1,\y+1.6);

  \def\x{-6.8-3}
  \def\y{+0.8-6}
  \draw[circle, -,dotted, very thick,rounded corners=8pt] (\x+0.1,\y+1.6)--(\x+0.1,\y+1.8)-- (\x+1.5,\y+1.8)-- (\x+1.5,\y-0.2)-- (\x+0.1,\y-0.2)--(\x+0.1,\y+1.6);
    \def\x{-6.8}
  \def\y{+0.8-6}
  \draw[circle, -,dotted, very thick,rounded corners=8pt] (\x+0.1,\y+1.6)--(\x+0.1,\y+1.8)-- (\x+1.5,\y+1.8)-- (\x+1.5,\y-0.2)-- (\x+0.1,\y-0.2)--(\x+0.1,\y+1.6);

\end{tikzpicture}
}
\end{center}
%$B(X,Y,Z)/X''/Y''/Z''_1/Z''_2$};
%\node at (\x-7.0,\y-8) {$B(X,Y,Z)/X'/Y'_1/Y'_2/Y'_3/Z'$
\caption{Decomposition of $B(X,Y,Z)\cong B(X,Y,Z)/X'_1/Y'_1/Y'_2/Y'_3/Z'_1\boxbackslash B(X,Y,Z)/X''_1/Y''_1/Z''_1/Z''_2$, $X'_1=\{u_{1}\},Y'_1=\{u_{2}\}, Y'_2=\{u_{3}\},Y'_3=\{u_{4}\},Z'_1=\{u_{5},u_{6}\},X''_1=\{u_{1}\},Y''_1=\{u_{2},u_3,u_4\}$, $Z''_1=\{u_{5}\},Z''_2=\{u_{6}\}$. The set $\zeta$ and the graph isomorphic to $B(X,Y,Z)$ induced by $\zeta$ in $V(B(X,Y,Z)/X'_1/Y'_1/Y'_2/Y'_3/Z_1'\boxtimes B(X,Y,Z)/X''_1/Y''_1/Z''_1/Z''_2)$ is indicated within the dotted region. 
%Note that the vertex $(\tilde{x}',\tilde{x}'')$ is the only vertex with level~0 in $B(X,Y,Z)/X'/Y'_1/Y'_2/Y'_3/Z'\Box B(X,Y,Z)/X''/Y''/Z''_1/Z''_2$, and, therefore, all vertices not in $\zeta$ will disappear from $B(X,Y,Z)/X'/Y'_1/Y'_2/Y'_3/Z'\boxtimes B(X,Y,Z)/X''/Y''/Z''_1/Z''_2$.
}
  \label{BiPartiteExampleDecomposition2}
\end{figure}
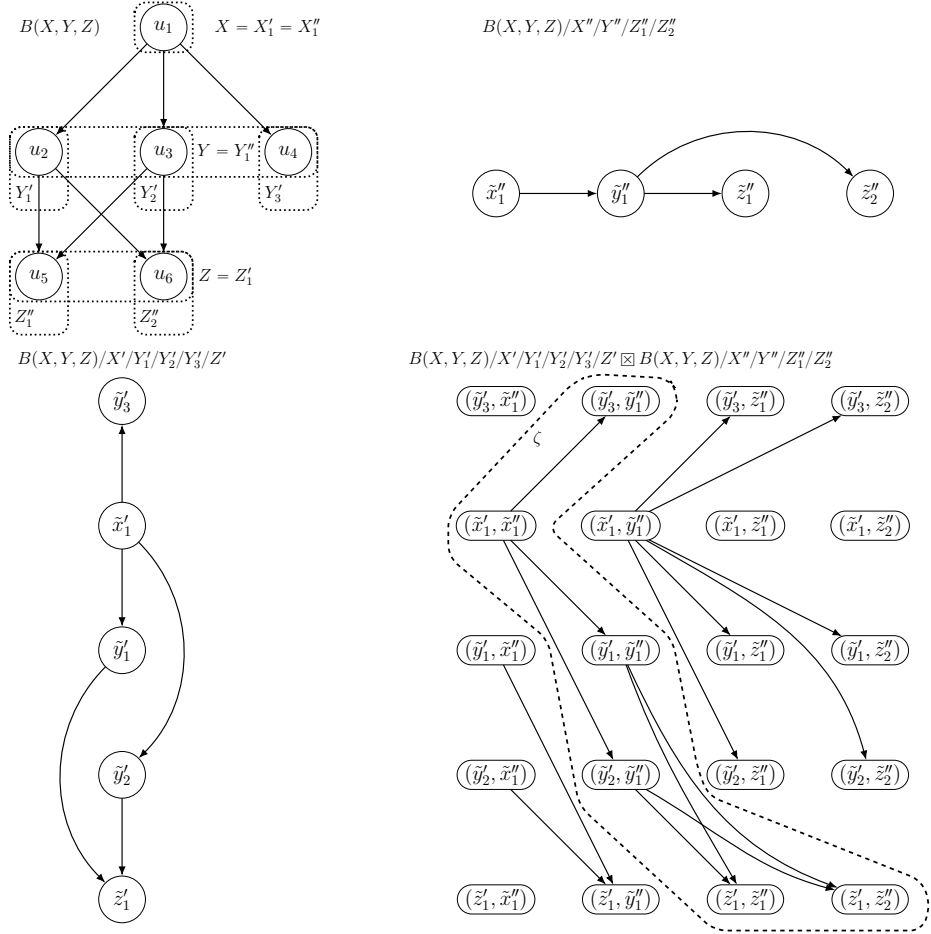
Before we continue with Lemma~\ref{lemma2} in which we decompose a 3-partite graph, we show in Figure~\ref{notdecomposable} that if the set of vertices $W''\subset Y$ of the bipartite subgraph $B(W'',Z)$ of $B(X,Y,Z)$ arc-induced by the arcs of $[Y,Z]$ is not a subset of the set of vertices $W'\subset Y$ of the bipartite subgraph $B(X,W')$ of $B(X,Y,Z)$ arc-induced by the arcs of $[X,Y]$, then this leads to the removal of the vertices $(\tilde{y}'_1,\tilde{y}''_1)$ and $(\tilde{y}'_2,\tilde{y}''_1)$ from $B(X,Y,Z)/X'/Y'_1/Y'_2/$ $Z'\boxtimes B(X,Y,Z)/X''/Y''_1/Y''_2/Y''_3/Z''$ (at the lower right of Figure~\ref{notdecomposable}) giving $B(X,Y,Z)/X'_1/Y'_1/Y'_2$ $/Z'_1\boxbackslash B(X,Y,Z)/X''_1/Y''_1/Y''_2/Y''_3/Z''_1$ (at the lower left of Figure~\ref{notdecomposable}) and, therefore, $B(X,Y,Z)\ncong B(X,Y,Z)/X'_1/Y'_1/Y'_2$ $/Z'_1\boxbackslash B(X,Y,Z)$ $/X''_1/Y''_1/Y''_2$ $/Y''_3/Z''_1$.
The vertices $(\tilde{y}_1,\tilde{y}''_1)$ and $(\tilde{y}_2,\tilde{y}''_1)$ are removed because they have $level>0$ in $B(X,Y,Z)$ $/X'_1/Y'_1/Y'_2/Z'_1\Box$ $B(X,Y,Z)/X''_1$ $/Y''_1/Y''_2/Y''_3/Z''_1$ and $level~0$ in $B(X,Y,Z)/X'_1/Y'_1$ $/Y'_2/Z'_1\boxtimes B(X,$ $Y,Z)/X''_1/Y''_1/Y''_2/Y''_3/Z''_1$.
%Hence $B(X,Y,Z)\ncong B(X,Y,Z)/X'/Y'_1/Y'_2/Z'\boxbackslash B(X,Y,Z)/X''/Y''_1/Y''_2/Y''_3/Z''$.
\begin{figure}[H]
\begin{center}
\resizebox{1.0\textwidth}{!}{
\begin{tikzpicture}[->,>=Latex,shorten >=0pt,auto,node distance=2.5cm,
  main node/.style={circle,fill=blue!10,draw, font=\sffamily\Large\bfseries}]%[scale=0.7]
  \tikzset{VertexStyle/.append style={%,fill=black,
  font=\itshape\large, shape = circle,inner sep = 2pt, outer sep = 0pt,minimum size = 32 pt,draw}}
%  \tikzset{EdgeStyle/.append style={thin}}
  \tikzset{LabelStyle/.append style={font = \itshape}}
  \SetVertexMath
  \def\x{0.0}
  \def\y{2.0}
\node at (\x-8.5,\y-0) {$B(X,Y,Z)$};
\node at (\x+14.5,\y-0) {$B(X,Y,Z)/X''_1/Y''_1/Y''_2/Y''_3/Z''_1$};
\node at (\x-1.5+6,\y-7.5) {$B(X,Y,Z)/X'_1/Y'_1/Y'_2/Z'_1$};
\node at (\x+16.5,\y-7.5) {$B(X,Y,Z)/X'_1/Y'_1/Y'_2/Z'_1\boxtimes B(X,Y,Z)/X''_1/Y''_1/Y''_2/Y''_3/Z''_1$};
\node at (\x+16.5-21,\y-20.5+13) {$B(X,Y,Z)/X'_1/Y'_1/Y'_2/Z'_1\boxbackslash B(X,Y,Z)/X''_1/Y''_1/Y''_2/Y''_3/Z''_1$};
\node at (\x+0.7,\y-0) {$X=X'_1=X''_1$};
\node at (\x+0.7,\y-6) {$Z=Z'_1=Z''_1$};
\node at (\x-10+8.5,\y-2.1) {$Y$};
\node at (\x-10+5.5,\y-2.5) {$Y'_1$};
\node at (\x-4+5.5,\y-2.5) {$Y'_2$};
\node at (\x-10+2.5,\y-3) {$Y''_1$};
\node at (\x-10+8.5,\y-3) {$Y''_2$};
\node at (\x-10+14.5,\y-3) {$Y''_3$};
\node at (\x+11,\y-10) {$\zeta$};

  \def\x{-10.0}
  \def\y{-2.0}
  \Vertex[x=\x+8.5, y=\y+4.0,L={u_{1}}]{u_1}
  \Vertex[x=\x+1, y=\y+1.0,L={u_{2}}]{u_2}
  \Vertex[x=\x+4, y=\y+1.0,L={u_{3}}]{u_3}
  \Vertex[x=\x+7, y=\y+1.0,L={u_{4}}]{u_4}
  \Vertex[x=\x+10, y=\y+1.0,L={u_{5}}]{u_5}
  \Vertex[x=\x+13, y=\y+1,L={u_{6}}]{u_6}
  \Vertex[x=\x+16, y=\y+1,L={u_{7}}]{u_7}
  \Vertex[x=\x+8.5, y=\y-2,L={u_{8}}]{u_8}

  \Edge(u_1)(u_4) 
  \Edge(u_1)(u_5) 
  \Edge(u_1)(u_6) 
  \Edge(u_1)(u_7) 
  \Edge(u_2)(u_8) 
  \Edge(u_3)(u_8) 
  \Edge(u_4)(u_8) 
  \Edge(u_5)(u_8) 

  \def\x{8.0}
  \def\y{-1.0}
  \Vertex[x=\x+3, y=\y+0.0,L={\tilde{x}''_1}]{u_1}
  \Vertex[x=\x+6, y=\y+0.0,L={\tilde{y}''_3}]{u_2}
  \Vertex[x=\x+9, y=\y+0.0,L={\tilde{y}''_2}]{u_3}
  \Vertex[x=\x+12, y=\y+0.0,L={\tilde{z}''_1}]{u_4}
  \Vertex[x=\x+15, y=\y+0.0,L={\tilde{y}''_1}]{u_5}

  \Edge(u_1)(u_2) 
  \Edge(u_3)(u_4) 
  \Edge[style={in = 135, out = 45,min distance=2cm}](u_1)(u_3) 
  \Edge(u_5)(u_4)

   \def\x{4.0}
  \def\y{-7.0}
  \Vertex[x=\x+0, y=\y-0.0,L={\tilde{x}'_1}]{u_1}
  \Vertex[x=\x+0, y=\y-3.0,L={\tilde{y}'_1}]{u_2}
  \Vertex[x=\x+0, y=\y-6.0,L={\tilde{y}'_2}]{u_3}
  \Vertex[x=\x+0, y=\y-9.0,L={\tilde{z}'_1}]{u_4}
%  \Vertex[x=\x+0, y=\y-12.0,L={\tilde{x}'_5}]{u_5}

  \Edge(u_1)(u_2) 
  \Edge(u_3)(u_4) 
  \Edge[style={in = 45, out = -45,min distance=2cm}](u_1)(u_3) 
  \Edge[style={in = 135, out = -135,min distance=2cm}](u_2)(u_4) 
%  \Edge(u_5)(u_4) 

  \tikzset{VertexStyle/.append style={%,fill=black,
  font=\itshape\large, shape = rounded rectangle, inner sep = 2pt, outer sep = 0pt,minimum size = 20 pt,draw}}

  \def\x{11.0}
  \def\y{-9.0}
  \Vertex[x=\x+0, y=\y+2.0,L={(\tilde{x}'_1,\tilde{x}''_1)}]{x1x1}
  \Vertex[x=\x+3, y=\y+2.0,L={(\tilde{x}'_1,\tilde{y}''_3)}]{x1x2}
  \Vertex[x=\x+6, y=\y+2.0,L={(\tilde{x}'_1,\tilde{y}''_2)}]{x1x3}
  \Vertex[x=\x+9, y=\y+2.0,L={(\tilde{x}'_1,\tilde{z}''_1)}]{x1x4}
  \Vertex[x=\x+12, y=\y+2.0,L={(\tilde{x}'_1,\tilde{y}''_1)}]{x1x5}

  \def\x{11.0}
  \def\y{-12.0}
  \Vertex[x=\x+0, y=\y+2.0,L={(\tilde{y}'_1,\tilde{x}''_1)}]{x2x1}
  \Vertex[x=\x+3, y=\y+2.0,L={(\tilde{y}'_1,\tilde{y}''_3)}]{x2x2}
  \Vertex[x=\x+6, y=\y+2.0,L={(\tilde{y}'_1,\tilde{y}''_2)}]{x2x3}
  \Vertex[x=\x+9, y=\y+2.0,L={(\tilde{y}'_1,\tilde{z}''_1)}]{x2x4}
  \Vertex[x=\x+12, y=\y+2.0,L={(\tilde{y}'_1,\tilde{y}''_1)}]{x2x5}

  \def\x{11.0}
  \def\y{-15.0}
  \Vertex[x=\x+0, y=\y+2.0,L={(\tilde{y}'_2,\tilde{x}''_1)}]{x3x1}
  \Vertex[x=\x+3, y=\y+2.0,L={(\tilde{y}'_2,\tilde{y}''_3)}]{x3x2}
  \Vertex[x=\x+6, y=\y+2.0,L={(\tilde{y}'_2,\tilde{y}''_2)}]{x3x3}
  \Vertex[x=\x+9, y=\y+2.0,L={(\tilde{y}'_2,\tilde{z}''_1)}]{x3x4}
  \Vertex[x=\x+12, y=\y+2.0,L={(\tilde{y}'_2,\tilde{y}''_1)}]{x3x5}

  \def\x{11.0}
  \def\y{-18.0}
  \Vertex[x=\x+0, y=\y+2.0,L={(\tilde{z}'_1,\tilde{x}''_1)}]{x4x1}
  \Vertex[x=\x+3, y=\y+2.0,L={(\tilde{z}'_1,\tilde{y}''_3)}]{x4x2}
  \Vertex[x=\x+6, y=\y+2.0,L={(\tilde{z}'_1,\tilde{y}''_2)}]{x4x3}
  \Vertex[x=\x+9, y=\y+2.0,L={(\tilde{z}'_1,\tilde{z}''_1)}]{x4x4}
  \Vertex[x=\x+12, y=\y+2.0,L={(\tilde{z}'_1,\tilde{y}''_1)}]{x4x5}

  \def\x{11.0}
  \def\y{-21.0}

  \Edge(x1x1)(x2x2) 
  \Edge(x1x1)(x2x3) 
  \Edge(x1x3)(x2x4) 
  \Edge(x1x5)(x2x4) 

  \Edge(x1x1)(x3x2) 
  \Edge[style={in = 105, out = -30,min distance=2cm}](x1x1)(x3x3) 
  \Edge(x1x3)(x3x4) 
  \Edge(x1x5)(x3x4) 

  \Edge[style={in = 120, out = -70,min distance=2cm}](x2x1)(x4x2) 
  \Edge[style={in = 160, out = -70,min distance=2cm}](x2x1)(x4x3) 
  \Edge(x2x3)(x4x4) 
  \Edge(x2x5)(x4x4) 

  \Edge(x3x1)(x4x2) 
  \Edge[style={in = 160, out = -45,min distance=2cm}](x3x1)(x4x3) 
  \Edge(x3x3)(x4x4) 
  \Edge(x3x5)(x4x4) 

\def\u{-20}
\def\v{13}
  \def\x{11.0+\u}
  \def\y{-22.0+\v}
  \Vertex[x=\x+0, y=\y+2.0,L={(\tilde{x}'_1,\tilde{x}''_1)}]{x1x1}

  \def\x{11.0+\u}
  \def\y{-25.0+\v}
  \Vertex[x=\x+3, y=\y+2.0,L={(\tilde{y}'_1,\tilde{y}''_3)}]{x2x2}
  \Vertex[x=\x+6, y=\y+2.0,L={(\tilde{y}'_1,\tilde{y}''_2)}]{x2x3}

  \def\x{11.0+\u}
  \def\y{-28.0+\v}
  \Vertex[x=\x+3, y=\y+2.0,L={(\tilde{y}'_2,\tilde{y}''_3)}]{x3x2}
  \Vertex[x=\x+6, y=\y+2.0,L={(\tilde{y}'_2,\tilde{y}''_2)}]{x3x3}

  \def\x{11.0+\u}
  \def\y{-31.0+\v}
  \Vertex[x=\x+9, y=\y+2.0,L={(\tilde{z}'_1,\tilde{z}''_1)}]{x4x4}

  \def\x{11.0+\u}
  \def\y{-34.0+\v}

  \Edge(x1x1)(x2x2) 
  \Edge(x1x1)(x2x3) 

  \Edge(x1x1)(x3x2) 
  \Edge[style={in = 105, out = -30,min distance=2cm}](x1x1)(x3x3) 

  \Edge(x2x3)(x4x4) 
  \Edge(x3x3)(x4x4) 

\node at (-11.0,-1.0) {$W''$};
\node at (8.0,-1.0) {$W'$};

  \def\x{-10.5}
  \def\y{+0.7-3}
  \draw[circle, -,dashed, very thick,rounded corners=8pt] (\x+0.1-1,\y+2.25)--(\x+0.1-1,\y+2.45)-- (\x+11.5,\y+2.45)-- (\x+11.5,\y+0.15)-- (\x+0.1-1,\y+0.15)--(\x+0.1-1,\y+2.25);

  \def\x{-3.0}
  \def\y{+0.6-3}
  \draw[circle, -,dashed, very thick,rounded corners=8pt] (\x+0.1-1,\y+2.5)--(\x+0.1-1,\y+2.7)-- (\x+11.5,\y+2.7)-- (\x+11.5,\y+0.1)-- (\x+0.1-1,\y+0.1)--(\x+0.1-1,\y+2.5);

  \def\x{6.75}
  \def\y{-11.25}
\draw[circle, -,dashed, very thick,rounded corners=8pt] (\x+3.1,\y+4.1)--(\x+3.1,\y+4.7)--(\x+3.6,\y+4.9)-- (\x+5.4,\y+4.8)-- (\x+5.6,\y+4.5)-- (\x+5.6,\y+4.0)-- (\x+7.0,\y+3.2)-- (\x+11.5,\y+1.8)-- (\x+11.6,\y-0.6)--(\x+11.9,\y-1.8)--(\x+12.2,\y-2.4)--(\x+14.4,\y-2.4)--(\x+15.1,\y-0.4+2.2)--(\x+17.6,\y-0.4+2.2)--(\x+17.6,\y-0.4+1.6)--(\x+17.6,\y-2.0)--(\x+14.0,\y-5.4)--(\x+12.5,\y-5.4)--(\x+11.5,\y-3.9)-- (\x+6.0,\y-2.2)-- (\x+6.0,\y-1.8)-- (\x+5.4,\y+1.8) --(\x+3.0,\y+3.8) --(\x+3.1,\y+4.1);
  \def\x{-2.3}
  \def\y{+0.8}
  \draw[circle, -,dotted, very thick,rounded corners=8pt] (\x+0.1,\y+1.6)--(\x+0.1,\y+1.8)-- (\x+1.5,\y+1.8)-- (\x+1.5,\y+0.6)-- (\x+0.1,\y+0.6)--(\x+0.1,\y+1.6);
  \def\x{-2.3}
  \def\y{+0.8-6}
  \draw[circle, -,dotted, very thick,rounded corners=8pt] (\x+0.1,\y+1.6)--(\x+0.1,\y+1.8)-- (\x+1.5,\y+1.8)-- (\x+1.5,\y+0.6)-- (\x+0.1,\y+0.6)--(\x+0.1,\y+1.6);
  \def\x{-6.8-3}
  \def\y{+0.8-3}
  \draw[circle, -,dotted, very thick,rounded corners=8pt] (\x+0.1-0.4,\y+2.0)--(\x+0.1-0.4,\y+2.2)-- (\x+1.5+6.4,\y+2.2)-- (\x+1.5+6.4,\y+0.2)-- (\x+0.1-0.4,\y+0.2)--(\x+0.1-0.4,\y+2.0);
  \def\x{-6.8-3+9}
  \def\y{+0.8-3}
  \draw[circle, -,dotted, very thick,rounded corners=8pt] (\x+0.1-0.4,\y+2.0)--(\x+0.1-0.4,\y+2.2)-- (\x+1.5+6.4,\y+2.2)-- (\x+1.5+6.4,\y+0.2)-- (\x+0.1-0.4,\y+0.2)--(\x+0.1-0.4,\y+2.0);
  
    \def\x{-2.3-7.5}
  \def\y{+0.8-3}
  \draw[circle, -,dotted, very thick,rounded corners=8pt] (\x+0.1,\y+1.6)--(\x+0.1,\y+1.8)-- (\x+4.5,\y+1.8)-- (\x+4.5,\y+0.6)-- (\x+0.1,\y+0.6)--(\x+0.1,\y+1.6);

    \def\x{-2.3-1.5}
  \def\y{+0.8-3}
  \draw[circle, -,dotted, very thick,rounded corners=8pt] (\x+0.1,\y+1.6)--(\x+0.1,\y+1.8)-- (\x+4.5,\y+1.8)-- (\x+4.5,\y+0.6)-- (\x+0.1,\y+0.6)--(\x+0.1,\y+1.6);

    \def\x{-2.3+4.5}
  \def\y{+0.8-3}
  \draw[circle, -,dotted, very thick,rounded corners=8pt] (\x+0.1,\y+1.6)--(\x+0.1,\y+1.8)-- (\x+4.5,\y+1.8)-- (\x+4.5,\y+0.6)-- (\x+0.1,\y+0.6)--(\x+0.1,\y+1.6);
  
    \def\x{-6.8-4+0.5}
  \def\y{+0.8-3}
  \draw[circle, -,dotted, very thick,rounded corners=8pt] (\x+0.1-0.4,\y+2.4)--(\x+0.1-0.4,\y+2.6)-- (\x+1.5+6.4+10,\y+2.6)-- (\x+1.5+6.4+10,\y+0.2-0.4)-- (\x+0.1-0.4,\y+0.2-0.4)--(\x+0.1-0.4,\y+2.4);

\end{tikzpicture}
}
\end{center}
\caption{The set of vertices $W''=\{u_2,u_3,u_4,u_5\}$ of the bipartite subgraph $B(W'',Z)$ of $B(X,Y,Z)$ arc-induced by the arcs of $[Y,Z]$ is not a subset of the set of vertices $W'=\{u_4,u_5,u_6,u_7\}$ of the bipartite subgraph $B(X,W')$ of $B(X,Y,Z)$ arc-induced by the arcs of $[X,Y]$. Although, the graph induced by the set of vertices $\zeta$ is isomorphic to $B(X,Y,Z)$, we have that $B(X,Y,Z)\ncong B(X,Y,Z)/X'_1/Y'_1/Y'_2/Z'_1\boxbackslash B(X,Y,Z)/X''_1/Y''_1/Y''_2/Y''_3/Z''_1$ (due to the removal of the vertices $(\tilde{y}'_1,\tilde{y}''_1)$ and $(\tilde{y}'_2,\tilde{y}''_1)$) by the VRSP.}
  \label{notdecomposable}
\end{figure}
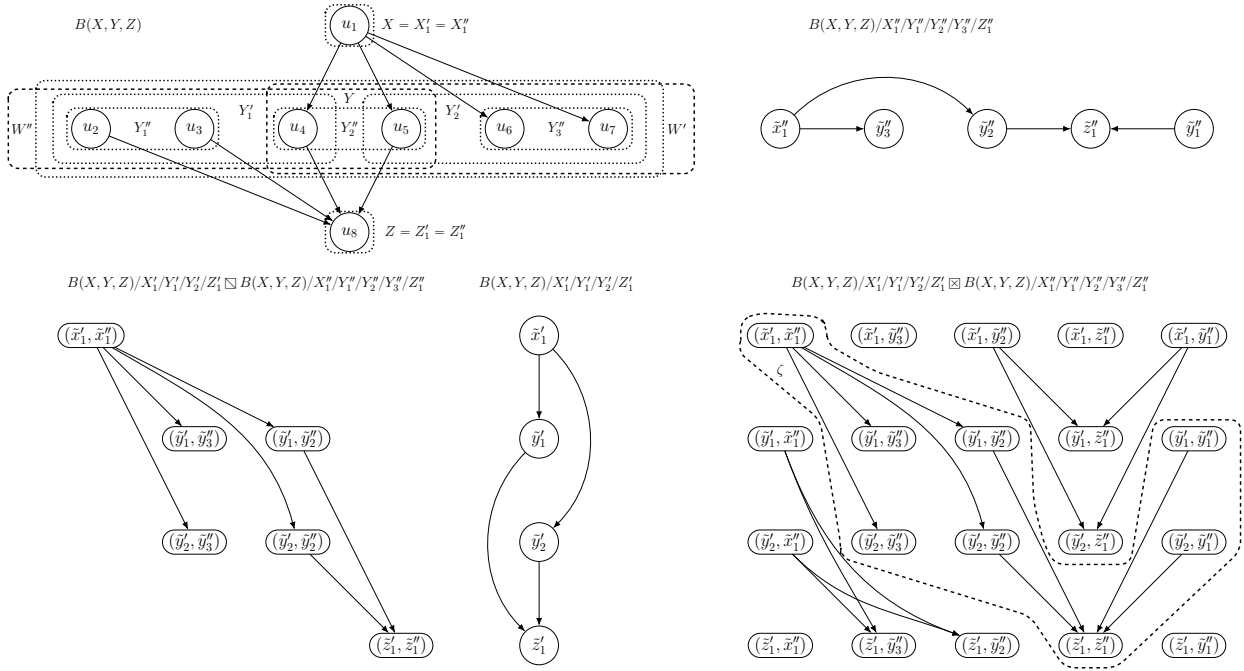
Furthermore, in Figure~\ref{notdecomposable1}, we show that when the requirement of Lemma~\ref{lemma2} that $[X,Y]$ and $[Y,Z]$ contain the only arcs of $B(X,Y,Z)$ is violated, we have  that the graph induced by the vertices of $\zeta$ is not isomorphic to $B(X,Y,Z)$. Hence, $B(X,Y,Z)\ncong B(X,Y,Z)/X'_1/Y'_{1}/Y'_{2}/Z'_1\boxbackslash B(X,Y,Z)/X''_1/Y''_{1}/Y''_{2}/Z''_1$. Note that $(\tilde{x}'_1,\tilde{y}''_2),(\tilde{y}'_1,\tilde{x}''_2),(\tilde{y}'_1,\tilde{x}''_1),(\tilde{y}'_2,\tilde{x}''_2)$ and $(\tilde{y}'_2,\tilde{x}''_1)$, and the arcs of which $(\tilde{x}'_1,\tilde{y}''_2),(\tilde{y}'_1,\tilde{x}''_2),(\tilde{y}'_1,\tilde{x}''_1),(\tilde{y}'_2,\tilde{x}''_2)$ and $(\tilde{y}'_2,\tilde{x}''_1)$ are an end are removed from $B(X,Y,Z)$ $/X'_1/Y'_{1}/Y'_{2}/Z'_1\boxtimes B(X,Y,Z)/X''_1$ $/Y''_{1}/Y''_{2}/Z''_1$ by the VRSP.

\begin{figure}[H]
\begin{center}
\resizebox{1.0\textwidth}{!}{
\begin{tikzpicture}[->,>=Latex,shorten >=0pt,auto,node distance=2.5cm,
  main node/.style={circle,fill=blue!10,draw, font=\sffamily\Large\bfseries}]%[scale=0.7]
  \tikzset{VertexStyle/.append style={%,fill=black,
  font=\itshape\large, shape = circle,inner sep = 2pt, outer sep = 0pt,minimum size = 32 pt,draw}}
%  \tikzset{EdgeStyle/.append style={thin}}
  \tikzset{LabelStyle/.append style={font = \itshape}}
  \SetVertexMath
  \def\x{2.0}
  \def\y{2.0}
\node at (\x-4.5,\y+1) {$B(X,Y,Z)$};
\node at (\x+10.5,\y-0) {$B(X,Y,Z)/X''_1/X''_2/Y''_{1}/Y''_{2}/Z''_1$};
\node at (\x-1.5,\y-9) {$B(X,Y,Z)/X'_1/Y'_{1}/Y'_{2}/Z'_1$};
\node at (\x+12.5,\y-7.5) {$B(X,Y,Z)/X'_1/Y'_{1}/Y'_{2}/Z'_1\boxtimes B(X,Y,Z)/X''_1/X''_2/Y''_{1}/Y''_{2}/Z''_1$};
%\node at (\x+0.7,\y-0) {$X=X'=X''$};
\node at (\x-3.0,\y-0) {$X=X'_1$};
\node at (\x-4.4,\y-0.85) {$X''_1$};
\node at (\x-1.5,\y-0.85) {$X''_2$};
\node at (\x+0.7,\y-6) {$Z=Z'_1=Z''_1$};
\node at (\x-10+6.1,\y-2) {$Y$};
\node at (\x-4+0.8,\y-3) {$Y'_1$};
\node at (\x+0.3,\y-3) {$Y'_2$};
\node at (\x-10+8.5,\y-2) {$Y''_1$};
\node at (\x-10+8.5,\y-4) {$Y''_2$};
\node at (\x+11,\y-10.5) {$\zeta$};

  \def\x{-8.0}
  \def\y{-2.0}
  \Vertex[x=\x+8.5, y=\y+4.0,L={u_{1,2}}]{u_1}
  \Vertex[x=\x+5.5, y=\y+4.0,L={u_{1,1}}]{u_2}
  \Vertex[x=\x+7, y=\y+0.0,L={u_{3,1}}]{u_4}
  \Vertex[x=\x+10, y=\y+0.0,L={u_{3,2}}]{u_5}
  \Vertex[x=\x+7, y=\y+2,L={u_{2,1}}]{u_6}
  \Vertex[x=\x+10, y=\y+2,L={u_{2,2}}]{u_7}
  \Vertex[x=\x+8.5, y=\y-2,L={u_{4,4}}]{u_8}

  \Edge[style={in = 180, out = -105,min distance=2cm}](u_2)(u_8) 

  \Edge(u_1)(u_4) 
  \Edge(u_1)(u_5) 
  \Edge(u_1)(u_6) 
  \Edge(u_1)(u_7) 
  \Edge(u_4)(u_8) 
  \Edge(u_5)(u_8) 

  \def\x{8.0}
  \def\y{-1.0}
  \Vertex[x=\x+3, y=\y+0.0,L={\tilde{x}_2''}]{u_1}
  \Vertex[x=\x+6, y=\y+0.0,L={\tilde{y}''_{1}}]{u_2}
  \Vertex[x=\x+9, y=\y+0.0,L={\tilde{y}''_{2}}]{u_3}
  \Vertex[x=\x+12, y=\y+0.0,L={\tilde{z}''_1}]{u_4}
  \Vertex[x=\x+15, y=\y+0.0,L={\tilde{x}_1''}]{u_5}

  \Edge(u_1)(u_2) 
  \Edge(u_3)(u_4) 
  \Edge[style={in = 135, out = 45,min distance=2cm}](u_1)(u_3) 
  \Edge(u_5)(u_4)

   \def\x{4.0}
  \def\y{-7.0}
  \Vertex[x=\x+0, y=\y-0.0,L={\tilde{x}'_1}]{u_1}
  \Vertex[x=\x+0, y=\y-3.0,L={\tilde{y}'_{1}}]{u_2}
  \Vertex[x=\x+0, y=\y-6.0,L={\tilde{y}'_{2}}]{u_3}
  \Vertex[x=\x+0, y=\y-9.0,L={\tilde{z}'_1}]{u_4}

  \Edge(u_1)(u_2) 
  \Edge(u_3)(u_4) 
  \Edge[style={in = 45, out = -45,min distance=2cm}](u_1)(u_3) 
  \Edge[style={in = 135, out = -135,min distance=2cm}](u_2)(u_4) 
  \Edge[style={in = 135, out = -135,min distance=2cm}](u_1)(u_4) 

  \tikzset{VertexStyle/.append style={%,fill=black,
  font=\itshape\large, shape = rounded rectangle, inner sep = 2pt, outer sep = 0pt,minimum size = 20 pt,draw}}

  \def\x{11.0}
  \def\y{-9.0}
  \Vertex[x=\x+0, y=\y+2.0,L={(\tilde{x}'_1,\tilde{x}_2'')}]{x1x1}
  \Vertex[x=\x+3, y=\y+2.0,L={(\tilde{x}'_1,\tilde{y}''_{1})}]{x1x2}
  \Vertex[x=\x+6, y=\y+2.0,L={(\tilde{x}'_1,\tilde{y}''_{2})}]{x1x3}
  \Vertex[x=\x+9, y=\y+2.0,L={(\tilde{x}'_1,\tilde{z}''_1)}]{x1x4}
  \Vertex[x=\x+12, y=\y+2.0,L={(\tilde{x}'_1,\tilde{x}_1'')}]{x1x5}

  \def\x{11.0}
  \def\y{-12.0}
  \Vertex[x=\x+0, y=\y+2.0,L={(\tilde{y}'_{_1},\tilde{x}_2'')}]{x2x1}
  \Vertex[x=\x+3, y=\y+2.0,L={(\tilde{y}'_{1},\tilde{y}''_{1})}]{x2x2}
  \Vertex[x=\x+6, y=\y+2.0,L={(\tilde{y}'_{1},\tilde{y}''_{2})}]{x2x3}
  \Vertex[x=\x+9, y=\y+2.0,L={(\tilde{y}'_{1},\tilde{z}''_1)}]{x2x4}
  \Vertex[x=\x+12, y=\y+2.0,L={(\tilde{y}'_{1},\tilde{x}_1'')}]{x2x5}

  \def\x{11.0}
  \def\y{-15.0}
  \Vertex[x=\x+0, y=\y+2.0,L={(\tilde{y}'_2,\tilde{x}_2'')}]{x3x1}
  \Vertex[x=\x+3, y=\y+2.0,L={(\tilde{y}'_{2},\tilde{y}''_1)}]{x3x2}
  \Vertex[x=\x+6, y=\y+2.0,L={(\tilde{y}'_{2},\tilde{y}''_{2})}]{x3x3}
  \Vertex[x=\x+9, y=\y+2.0,L={(\tilde{y}'_{2},\tilde{z}''_1)}]{x3x4}
  \Vertex[x=\x+12, y=\y+2.0,L={(\tilde{y}'_{2},\tilde{x}_1'')}]{x3x5}

  \def\x{11.0}
  \def\y{-18.0}
  \Vertex[x=\x+0, y=\y+2.0,L={(\tilde{z}'_1,\tilde{x}_2'')}]{x4x1}
  \Vertex[x=\x+3, y=\y+2.0,L={(\tilde{z}'_1,\tilde{y}''_1)}]{x4x2}
  \Vertex[x=\x+6, y=\y+2.0,L={(\tilde{z}'_1,\tilde{y}''_2)}]{x4x3}
  \Vertex[x=\x+9, y=\y+2.0,L={(\tilde{z}'_1,\tilde{z}''_1)}]{x4x4}
  \Vertex[x=\x+12, y=\y+2.0,L={(\tilde{z}'_1,\tilde{x}_1'')}]{x4x5}

  \def\x{11.0}
  \def\y{-21.0}

  \Edge(x1x1)(x2x2) 
  \Edge(x1x1)(x2x3) 
  \Edge[style={in =175, out = -130,min distance=2cm}](x1x1)(x4x2) 
  \Edge[style={in =120, out = -60,min distance=7cm}](x1x1)(x4x3) 
  \Edge(x1x3)(x2x4) 
  \Edge(x1x5)(x2x4) 

  \Edge(x1x1)(x3x2) 
  \Edge[style={in = 105, out = -30,min distance=2cm}](x1x1)(x3x3) 
  \Edge(x1x3)(x3x4) 
  \Edge[style={in = 115, out = -65,min distance=2cm}](x1x3)(x4x4) 
  \Edge(x1x5)(x3x4) 

  \Edge[style={in = 120, out = -70,min distance=2cm}](x2x1)(x4x2) 
  \Edge[style={in = 160, out = -70,min distance=2cm}](x2x1)(x4x3) 
  \Edge(x2x3)(x4x4) 
%  \Edge(x2x5)(x3x4) 
  \Edge(x2x5)(x4x4) 
  \Edge(x1x5)(x4x4) 

  \Edge(x3x1)(x4x2) 
  \Edge[style={in = 160, out = -45,min distance=2cm}](x3x1)(x4x3) 
  \Edge(x3x3)(x4x4) 
  \Edge(x3x5)(x4x4) 

  \def\x{6.75}
  \def\y{-11.25}
\draw[circle, -,dashed, very thick,rounded corners=8pt] (\x+3.1,\y+4.1)--(\x+3.1,\y+4.7)--(\x+3.6,\y+4.9)-- (\x+5.4,\y+4.8)-- (\x+5.6,\y+4.5)-- (\x+5.6,\y+4.0)-- (\x+7.0,\y+3.2)-- (\x+11.5,\y+1.8)-- (\x+14.5,\y+1.8)--(\x+15.2,\y+4.9)--(\x+17.5,\y+4.9)-- (\x+17.5,\y+3.9)-- (\x+15.0,\y-0.6+2)--(\x+14.7,\y-1.8)--(\x+14.4,\y-5.4)--(\x+6.0,\y-5.4)--(\x+6.0,\y-3.9)--  (\x+5.4,\y+1.8) --(\x+3.0,\y+3.8) --(\x+3.1,\y+4.1);
  \def\x{-2.3+2}
  \def\y{+0.8}
  \draw[circle, -,dotted, very thick,rounded corners=8pt] (\x+0.1,\y+1.6)--(\x+0.1,\y+1.8)-- (\x+1.5,\y+1.8)-- (\x+1.5,\y+0.1)-- (\x+0.1,\y+0.1)--(\x+0.1,\y+1.6);
    \def\x{-2.3-1}
  \def\y{+0.8}
  \draw[circle, -,dotted, very thick,rounded corners=8pt] (\x+0.1,\y+1.6)--(\x+0.1,\y+1.8)-- (\x+1.5,\y+1.8)-- (\x+1.5,\y+0.1)-- (\x+0.1,\y+0.1)--(\x+0.1,\y+1.6);
    \def\x{-2.3+2}
  \def\y{+0.8-6}
  \draw[circle, -,dotted, very thick,rounded corners=8pt] (\x+0.1,\y+1.6)--(\x+0.1,\y+1.8)-- (\x+1.5,\y+1.8)-- (\x+1.5,\y+0.6)-- (\x+0.1,\y+0.6)--(\x+0.1,\y+1.6);

    \def\x{-2.3-1.5+2}
  \def\y{+0.8-4}
  \draw[circle, -,dotted, very thick,rounded corners=8pt] (\x+0.1,\y+1.6+2)--(\x+0.1,\y+1.8+2)-- (\x+1.5,\y+1.8+2)-- (\x+1.5,\y+0.6)-- (\x+0.1,\y+0.6)--(\x+0.1,\y+1.6+2);

     \def\x{-0.3-0.5+2}
  \def\y{+0.8-4}
  \draw[circle, -,dotted, very thick,rounded corners=8pt] (\x+0.1,\y+1.6+2)--(\x+0.1,\y+1.8+2)-- (\x+1.5,\y+1.8+2)-- (\x+1.5,\y+0.6)-- (\x+0.1,\y+0.6)--(\x+0.1,\y+1.6+2);

    \def\x{-2.3-1.5+2}
  \def\y{+0.8-2}
  \draw[circle, -,dotted, very thick,rounded corners=8pt] (\x+0.1,\y+1.6)--(\x+0.1,\y+1.8)-- (\x+4.5,\y+1.8)-- (\x+4.5,\y+0.6)-- (\x+0.1,\y+0.6)--(\x+0.1,\y+1.6);
    \def\x{-2.3-1.0}
  \def\y{+0.8-0}
  \draw[circle, -,dotted, very thick,rounded corners=8pt] (\x+0.1,\y+1.6)--(\x+0.1,\y+1.8)-- (\x+4.5,\y+1.8)-- (\x+4.5,\y+0.6)-- (\x+0.1,\y+0.6)--(\x+0.1,\y+1.6);

    \def\x{-2.3-1.5+2}
  \def\y{+0.8-4}
  \draw[circle, -,dotted, very thick,rounded corners=8pt] (\x+0.1,\y+1.6)--(\x+0.1,\y+1.8)-- (\x+4.5,\y+1.8)-- (\x+4.5,\y+0.6)-- (\x+0.1,\y+0.6)--(\x+0.1,\y+1.6);
    \def\x{-0.3-4+0.5+2}
  \def\y{+0.8-3}
  \draw[circle, -,dotted, very thick,rounded corners=8pt] (\x+0.1-0.4,\y+3.4-0.5)--(\x+0.1-0.4,\y+3.6-0.5)-- (\x+1.5+4.4-1,\y+3.6-0.5)-- (\x+1.5+4.4-1,\y+0.2-1.4+0.5)-- (\x+0.1-0.4,\y+0.2-1.4+0.5)--(\x+0.1-0.4,\y+3.4-0.5);

\end{tikzpicture}
}
\end{center}
\caption{The requirement that $[X,Y]$ and $[Y,Z]$ contain the only arcs of $B(X,Y,Z)$ is violated, giving that the graph induced by the vertices of $\zeta$ is not isomorphic to $B(X,Y,Z)$. Hence, $B(X,Y,Z)\ncong B(X,Y,Z)/X'/Y'_{1}/Y'_{2}/Z'\boxbackslash B(X,Y,Z)/X''/Y''_{1}/Y''_{2}/Z''$. Furthermore, $(\tilde{x}'_1,\tilde{y}''_2),(\tilde{y}'_1,\tilde{x}''_2),(\tilde{y}'_1,\tilde{x}''_1),(\tilde{y}'_2,\tilde{x}''_2)$ and $(\tilde{y}'_2,\tilde{x}''_1)$, and the arcs of which $(\tilde{x}'_1,\tilde{y}''_2),(\tilde{y}'_1,\tilde{x}''_2),(\tilde{y}'_1,\tilde{x}''_1),(\tilde{y}'_2,\tilde{x}''_2)$ and $(\tilde{y}'_2,\tilde{x}''_1)$ are an end are removed from $B(X,Y,Z)/X'/Y'_{1}/Y'_{2}/Z'\boxtimes B(X,Y,Z)/X''$ $/Y''_{1}/Y''_{2}/Z''$ by the VRSP.}
  \label{notdecomposable1}
\end{figure}
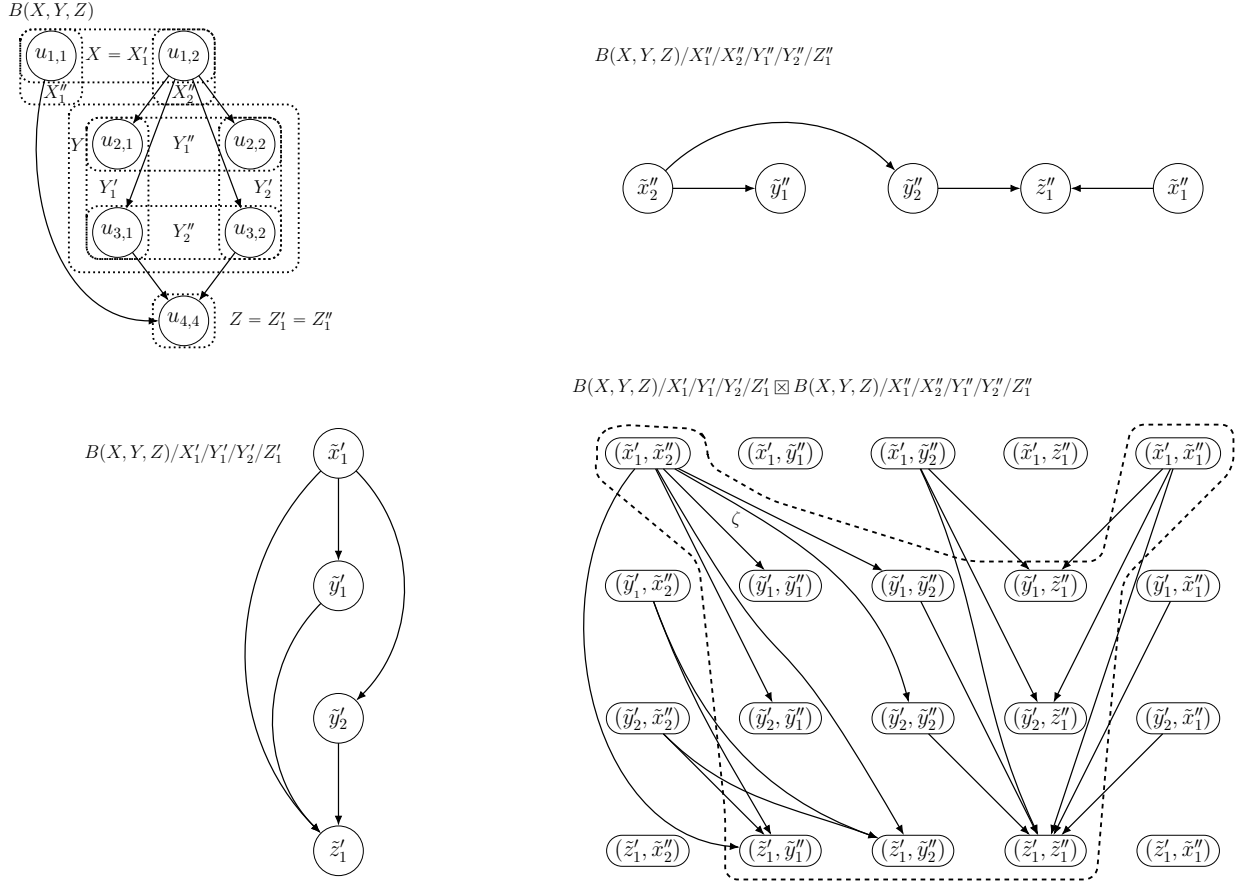
The requirement that $[X,Y]$ and $[Y,Z]$ contain only forward arcs, or $[X,Y]$ and $[Y,Z]$ contain only backward arcs must not be violated.
Otherwise, for example, if $[X,Y]$ contains only forward arcs and $[Y,Z]$ contains only backward arcs, we have a bipartite graph $B(X\cup Z,Y)$ where $[X\cup Z,Y]$ contains only forward arcs.
Because such a graph is not complete it is also not decomposable by Lemma~6.1 in  \cite{mod-arxiv} as is shown by the example given in Figure~4 in \cite{mod-arxiv}.
Now, for the decomposition of a 3-partite graph $B(X,Y,Z)$, we have the following lemma.
\begin{lemma}\label{lemma2}
Let $B(X,Y,Z)$ be a weakly connected 3-partite graph where the labels of all arcs are the same.
Let $[X,Y]$ and $[Y,Z]$ contain the only arcs of $B(X,Y,Z)$.
Let $[X,Y]$ and $[Y,Z]$ contain only forward arcs, or let $[X,Y]$ and $[Y,Z]$ contain only backward arcs.
Let $B(X,Y)$ be a complete bipartite subgraph of $B(X,Y,Z)$ induced by $X\cup Y$.
Let $B(W,Z), W\subseteq Y,$ be a complete bipartite subgraph of $B(X,Y,Z)$ arc-induced by all arcs of $[Y,Z]$.
%Let $c_1,\ldots,c_8\in~\hspace{-5pt}\mathbb{N}$.
Let $|X|=c_1\cdot c_2,GCD(|Y|,|W|)=c_3,|Y|=c_3\cdot c_4,|W|=c_3\cdot c_7,|Z|=c_5\cdot c_6,c_1,\ldots,c_7\in~\hspace{-5pt}\mathbb{N}$.
%
%Let $GCD(|Y|,|W|)=c_3$, $|Y|=c_3\cdot c_4$ and $|W|=c_5\cdot c_6$, $c_5\leq c_3,c_6\leq c_4$.
%Let $Y$ be partitioned into subsets $Y'_i$ of equal size $|Y'_i|=c_4,i=1,\ldots,c_3,$ and let $W$ be partitioned into subsets $W'_i$ of equal size $|W'_i|=c_6$ such that $W'_i\subseteq Y'_i$ for $i=1,\ldots,c_5$, and let $Y$ be partitioned into subsets $Y''_i,i=1,\ldots,c_4,$ of equal size $|Y''_i|=c_3$ and let $W$ be partitioned into subsets $W''_i$ of equal size $|W''_i|=c_5$ such that $W''_i\subseteq Y''_i$ for $i=1,\ldots,c_6$, together with $|Y'_i\cap Y''_j|=1,i\in\{1,\ldots,c_3\},j\in\{1,\ldots,c_4\}$.
%Let $|Z|=c_7\cdot c_8$.
Then there exist $X'_g,X''_h, Y'_i,Y''_j,Z'_k,Z''_l,$ such that $B(X,Y,Z)\cong B(X,Y,Z)\overunderslash{g=1}{c_1}X'_g\overunderslash{i=1}{c_3}Y'_i\overunderslash{k=1}{c_5}Z'_k\boxbackslash B(X,Y,Z)\overunderslash{h=1}{c_2}X''_h\overunderslash{j=1}{c_4}Y''_j\overunderslash{l=1}{c_6}Z''_l$.
\end{lemma}
\begin{proof}
Let $[X,Y]$ and $[Y,Z]$ contain only forward arcs.
%Let $g\in\{1,\ldots,c_1\},h\in\{1,\ldots,c_2\},i\in\{1,\ldots,c_3\},j\in\{1,\ldots,c_4\},k\in\{1,\ldots,c_7\},l\in\{1,\ldots,c_8\}$.
It is sufficient to define a mapping $\phi:V(B(X,Y,Z))\rightarrow V(B(X,Y,Z)\overunderslash{g=1}{c_1}X'_g\overunderslash{i=1}{c_3}Y'_i\overunderslash{k=1}{c_5}Z'_k\,\boxbackslash B(X,Y,Z)\overunderslash{h=1}{c_2}X''_h\overunderslash{j=1}{c_4}Y''_j\overunderslash{l=1}{c_6}Z''_l)$ and to prove that $\phi$ is an isomorphism from $B(X,Y,Z)$ to $B(X,Y,Z)\overunderslash{g=1}{c_1}X'_g\overunderslash{i=1}{c_3}Y'_i\overunderslash{k=1}{c_5}Z'_k\boxbackslash B(X,Y,Z) \overunderslash{h=1}{c_2}X''_h$ $\overunderslash{j=1}{c_4}Y''_j\overunderslash{l=1}{c_6}Z''_l$.

%Without loss of generality, we have the following definitions for $V(B(X,Y,Z))$.
Because, $|X|=c_1\cdot c_2$, we define $X=\{u_{1,1},\ldots,u_{1,c_{_2}},\ldots,u_{c_{_1},1},\ldots,u_{c_{_1},c_{_2}}\}$.
Then, we can contract $X$ using the sets $X'_1,\ldots, X'_{c_{_1}}$, $X'_g=\{u_{g,1},\ldots,u_{g,c_{_2}}\}$, $|X'_g|=c_2,g=1,\ldots,c_1$. The vertices in the sets $X'_1,\ldots, X'_{c_{_1}}$ are then replaced by the vertices $\tilde{x}'_1,\ldots, \tilde{x}'_{c_{_1}}$, respectively, (note that there are no arcs that have both their ends in $X'_g$), and we can contract $X$ using the sets $X''_{1},\ldots,X''_{c_{_2}},$ $X''_h=\{u_{1,h},\ldots,u_{c_{_1},h}\}$, $|X''_h|=c_1,h=1,\ldots,c_2$. The vertices in the sets $X''_1,\ldots, X''_{c_{_2}}$ are then replaced by the vertices $\tilde{x}''_1,\ldots, \tilde{x}''_{c_{_2}}$, respectively, (note that there are no arcs that have both their ends in  $X''_h$).

Because, $|Y|=c_3\cdot c_4$, $|W|=c_3\cdot c_7$ and $W\subseteq Y$ we have that $c_7\leq c_4$.
Therefore, we define  $Y=\{v_{1,1},\ldots,v_{1,c_{_4}},\ldots, v_{c_{_3},1},\ldots,$ $v_{c_{_3},c_{_4}}\}$ and $W=\{v_{1,1},\ldots,v_{1,c_{_7}},\ldots, v_{c_{_3},1},\ldots,v_{c_{_3},c_{_7}}\}$, satisfying $W\subseteq Y$.
Then, we can contract $Y$ using the sets $Y'_1,\ldots,Y'_{c_{_3}}$, $Y'_i=\{v_{i,1},\ldots,v_{i,c_{_4}}\}, |Y'_i|=c_4,i=1,\ldots,c_3$. The vertices in the sets $Y'_1,\ldots, Y'_{c_{_3}}$ are then replaced by the vertices $\tilde{y}'_1,\ldots, \tilde{y}'_{c_{_3}}$, respectively, (note that there are no arcs that have both their ends in $Y'_i$), and we can contract $Y$ using the sets  $Y''_1,\ldots,Y''_{c_{_4}}$, $Y''_j=\{v_{1,j},\ldots,v_{c_{_3},j}\},|Y''_j|=c_3,j=1,\ldots,c_4$. The vertices in the sets $Y''_1,\ldots, Y''_{c_{_4}}$ are then replaced by the vertices $\tilde{y}''_1,\ldots, \tilde{y}''_{c_{_4}}$, respectively, (note that there are no arcs that have both their ends in $Y'_h$). 
Using analogue definitions for $W'_i$ and $W''_j$ with respect to the definitions of $Y'_i$ and $Y''_j$, we have that the vertices in the sets $W'_1,\ldots, W'_{c_{_3}}$ are replaced by the vertices $\tilde{y}'_1,\ldots, \tilde{y}'_{c_{_3}}$ and the vertices in the sets $W''_1,\ldots, W''_{c_{_7}}$ are replaced by the vertices $\tilde{y}''_1,\ldots, \tilde{y}''_{c_{_7}}$.

Because, $|Z|=c_5\cdot c_6$,  we define $Z=\{w_{1,1},\ldots,w_{1,c_{_6}},\ldots,w_{c_{_5},1},\ldots,w_{c_{_5},c_{_6}}\}$.
Then, we can contract $Z$ using the sets $Z'_1,\ldots, Z'_{c_{_5}}$, $Z'_k=\{w_{k,1},\ldots,w_{k,c_{_6}}\}$, $|Z'_k|=c_6,k=1,\ldots,c_5$. 
The vertices in the sets $Z'_1,\ldots, Z'_{c_{_5}}$ are then replaced by the vertices $\tilde{z}'_1,\ldots, \tilde{z}'_{c_{_5}}$, respectively, (note that there are no arcs that have both their ends in $Z'_k$), and we can contract $Z$ using the sets $Z''_{1},\ldots,Z''_{c_{_6}},$ $Z''_l=\{w_{1,l},\ldots,w_{c_{_5},l}\}$, $|Z'_l|=c_5,l=1,\ldots,c_6$. The vertices in the sets $Z''_1,\ldots, Z''_{c_{_6}}$ are then replaced by the vertices $\tilde{z}''_1,\ldots, \tilde{z}''_{c_{_6}}$, respectively, (note that there are no arcs that have both their ends in  $Z''_l$). 
 
Consider the mapping $\phi:V(B(X,Y,Z))\rightarrow V(B(X,Y$ $,Z)\overunderslash{g=1}{c_1}X'_g\overunderslash{i=1}{c_3}Y'_i\overunderslash{k=1}{c_5}Z'_k\boxbackslash B(X,Y,Z)$ $\overunderslash{h=1}{c_2}X''_h\overunderslash{j=1}{c_4}Y''_j\overunderslash{l=1}{c_6}Z''_l)$ defined by $\phi(u_{g,h})=(\tilde{x}'_g,\tilde{x}''_h),\phi(v_{i,j})=(\tilde{y}'_i,\tilde{y}''_j),\phi(w_{k,l})=(\tilde{z}'_k,\tilde{z}''_l)$.
Then $\phi$ is obviously a bijective map if $V(B(X,Y,Z)\overunderslash{g=1}{c_1}X'_g\overunderslash{i=1}{c_3}Y'_i\overunderslash{k=1}{c_5}Z'_k\boxbackslash B(X,Y,Z)\overunderslash{h=1}{c_2}X''_h\overunderslash{j=1}{c_4}Y''_j\overunderslash{l=1}{c_6}Z''_l)=\zeta$, where $\zeta$ is defined as $\zeta=\{(\tilde{x}'_g,\tilde{x}''_h)\mid u_{g,h}\in X,\phi(u_{g,h})=(\tilde{x}'_g,\tilde{x}''_h)\}\cup \{(\tilde{y}'_i,\tilde{y}''_j)\mid v_{i,j}\in Y,\phi(v_{i,j})=(\tilde{y}'_i,\tilde{y}''_j)\}\cup \{(\tilde{z}'_k,\tilde{z}''_l)\mid w_{k,l}\in Z,\phi(w_{k,l})=(\tilde{z}'_k,\tilde{z}''_l)\}$.

We are going to show this later by arguing all the other vertices (and their labelled arcs) $(\tilde{x}'_g,\tilde{y}''_j), (\tilde{x}'_g,\tilde{z}''_l),(\tilde{y}'_i,\tilde{x}''_h),(\tilde{y}'_i,\tilde{z}''_l),(\tilde{z}'_k,\tilde{x}''_h)$ and $(\tilde{z}'_k,\tilde{y}''_j)$ of $B(X,Y,Z)\overunderslash{g=1}{c_1}X'_g\overunderslash{i=1}{c_3}Y'_i\overunderslash{k=1}{c_5}Z'_k\boxtimes B(X,Y,$ $Z)\overunderslash{h=1}{c_2}X''_h\overunderslash{j=1}{c_4}Y''_j\overunderslash{l=1}{c_6}Z''_l$ will disappear from $B(X,Y,Z)\overunderslash{g=1}{c_1}X'_g\overunderslash{i=1}{c_3}Y'_i\overunderslash{k=1}{c_5}Z'_k\boxtimes B(X,Y,Z)\overunderslash{h=1}{c_2}X''_h\overunderslash{j=1}{c_4}Y''_j$ $\overunderslash{l=1}{c_6}Z''_l$. 	
But first we are going to prove the following claim. 
%\newpage
\begin{claim}\label{claim3}
The subgraph of $B(X,Y,Z)\overunderslash{g=1}{c_1}X'_g\overunderslash{i=1}{c_3}Y'_i\overunderslash{k=1}{c_5}Z'_k\boxtimes B(X,Y,Z)\overunderslash{h=1}{c_2}X''_h\overunderslash{j=1}{c_4}Y''_j\overunderslash{l=1}{c_6}Z''_l$ induced by $\zeta$ is isomorphic to $B(X,Y,Z)$.
\end{claim}
\begin{proof}
$\phi$ is a bijection from $V(B(X,Y,Z))$ to $\zeta$.
It remains to show that this bijection preserves the arcs and their label pairs.
Due to Lemma~\ref{lemma1} and because the subgraph $B(X,Y)$ is a complete bipartite subgraph, we have that an arc $u_{g,h}v_{i,j}$ in $[X,Y]$ is represented by the arc $(\tilde{x}'_g,\tilde{x}''_h)(\tilde{y}'_i,\tilde{y}''_j)$ in $A(B(X,Y,Z)\overunderslash{g=1}{c_1}X'_g\overunderslash{i=1}{c_3}Y'_i\overunderslash{k=1}{c_5}Z'_k\boxtimes B(X,Y,Z)\overunderslash{h=1}{c_2}X''_h\overunderslash{j=1}{c_4}Y''_j\overunderslash{l=1}{c_6}Z''_l)$ and, due to Lemma~\ref{lemma1} and because the subgraph $B(W,Z)$ is a complete bipartite subgraph, an arc $v_{i,j}w_{k,l}$ in $[W,Z]$ is represented by the arc $(\tilde{y}'_i,\tilde{y}''_j)(\tilde{z}'_k,\tilde{z}''_l)$ in $A(B(X,Y,Z)\overunderslash{g=1}{c_1}X'_g\overunderslash{i=1}{c_3}Y'_i\overunderslash{k=1}{c_5}Z'_k\boxtimes B(X,Y,Z)$ $\overunderslash{h=1}{c_2}X''_h\overunderslash{j=1}{c_4}Y''_j\overunderslash{l=1}{c_6}Z''_l)$.
Together with $[Y\backslash W,Z]=\emptyset$, we have that the subgraph of $B(X,Y,Z)\overunderslash{g=1}{c_1}X'_g$ $\overunderslash{i=1}{c_3}Y'_i\overunderslash{k=1}{c_5}Z'_k\boxtimes B(X,Y,Z)\overunderslash{h=1}{c_2}X''_h\overunderslash{j=1}{c_4}Y''_j\overunderslash{l=1}{c_6}Z''_l$ induced by $\zeta$ is isomorphic to $B(X,Y,Z)$.
\end{proof}
We continue with the proof of Lemma~\ref{lemma2}. 
It remains to show that all other vertices of $B(X,Y,Z)\overunderslash{g=1}{c_1}X'_g$ $\overunderslash{i=1}{c_3}Y'_i\overunderslash{k=1}{c_5}Z'_k\boxtimes B(X,Y,Z)\overunderslash{h=1}{c_2}X''_h\overunderslash{j=1}{c_4}Y''_j\overunderslash{l=1}{c_6}Z''_l$, except for the vertices of $\zeta$, disappear from $B(X,Y,Z)\overunderslash{g=1}{c_1}X'_g$ $\overunderslash{i=1}{c_3}Y'_i\overunderslash{k=1}{c_5}Z'_k\boxtimes B(X,Y,Z)\overunderslash{h=1}{c_2}X''_h\overunderslash{j=1}{c_4}Y''_j\overunderslash{l=1}{c_6}Z''_l$.
Due to Lemma~\ref{lemma1}, this is clear for the vertices $(\tilde{x}'_g,\tilde{y}''_j)$ and $(\tilde{y}'_i,\tilde{x}''_h)$.
Likewise, due to Lemma~\ref{lemma1} and the removal of all $(\tilde{x}'_g,\tilde{y}''_j)$  and $(\tilde{y}'_i,\tilde{x}''_h)$) , this is also clear for the vertices $(\tilde{y}'_i,\tilde{z}''_l)$ and $(\tilde{z}'_k,\tilde{y}''_j)$.
Remains to show that the vertices $(\tilde{x}'_g,\tilde{z}''_l)$ and $(\tilde{z}'_k,\tilde{x}''_h)$ are also removed from $B(X,Y,Z)\overunderslash{g=1}{c_1}X'_g$ $\overunderslash{i=1}{c_3}Y'_i\overunderslash{k=1}{c_5}Z'_k\boxtimes B(X,Y,Z)\overunderslash{h=1}{c_2}X''_h\overunderslash{j=1}{c_4}Y''_j\overunderslash{l=1}{c_6}Z''_l$.
Because there are no arcs $a$ with $ head(a)=\tilde{x}'_g$ in $B(X,Y,Z)\overunderslash{g=1}{c_1}X'_g$ $\overunderslash{i=1}{c_3}Y'_i\overunderslash{k=1}{c_5}Z'_k$ there are no arcs $b$ with $head(b)=(\tilde{x}'_g,\tilde{z}''_l)$  in $B(X,Y,Z)\overunderslash{g=1}{c_1}X'_g$ $\overunderslash{i=1}{c_3}Y'_i\overunderslash{k=1}{c_5}Z'_k\boxtimes B(X,Y,$ $Z)\overunderslash{h=1}{c_2}X''_h\overunderslash{j=1}{c_4}Y''_j\overunderslash{l=1}{c_6}Z''_l$, and, because there are no arcs $a$ with $head(a)= \tilde{x}''_h$ in $B(X,Y,Z)\overunderslash{h=1}{c_2}X''_h\overunderslash{j=1}{c_4}$ $Y''_j\overunderslash{l=1}{c_6}Z''_l\,$ there are no arcs $b$ with $head(b)=(\tilde{z}'_k,\tilde{x}''_h)$ in $B(X,Y,Z)\overunderslash{g=1}{c_1}X'_g\overunderslash{i=1}{c_3}Y'_i\overunderslash{k=1}{c_5}Z'_k\boxtimes B(X,Y,Z)$ $\overunderslash{h=1}{c_2}X''_h$ $\overunderslash{j=1}{c_4}Y''_j\overunderslash{l=1}{c_6}Z''_l$.
Hence, the $(x'_g,z''_l)$ and $(\tilde{z}'_k,\tilde{x}''_h)$ must have level~0 in $B(X,Y,Z)\overunderslash{g=1}{c_1}X'_g\overunderslash{i=1}{c_3}Y'_i$ $\overunderslash{k=1}{c_5}Z'_k\boxtimes B(X,Y,Z)\overunderslash{h=1}{c_2}X''_h\overunderslash{j=1}{c_4}Y''_j\overunderslash{l=1}{c_6}Z''_l$.
But, the level of $(\tilde{x}'_g,\tilde{z}''_l)$ and $(\tilde{z}'_k,\tilde{x}''_h)$ is greater than zero in $B(X,Y,$ $Z)\overunderslash{g=1}{c_1}X'_g$ $\overunderslash{i=1}{c_3}Y'_i\overunderslash{k=1}{c_5}Z'_k\Box$ $B(X,Y,Z)\overunderslash{h=1}{c_2}X''_h\overunderslash{j=1}{c_4}Y''_j\overunderslash{l=1}{c_6}Z''_l$ and the level of $(\tilde{x}'_g,\tilde{z}''_l)$ and $(\tilde{z}'_k,\tilde{x}''_h)$ is zero in $B(X,Y,Z)\overunderslash{g=1}{c_1}X'_g\overunderslash{i=1}{c_3}Y'_i$ $\overunderslash{k=1}{c_5}Z'_k\boxtimes B(X,Y,Z)\overunderslash{h=1}{c_2}X''_h\overunderslash{j=1}{c_4}Y''_j$ $\overunderslash{l=1}{c_6}Z''_l$.
Therefore, $(\tilde{x}'_g,\tilde{z}''_l)$ and $(\tilde{z}'_k,\tilde{x}''_h)$ are removed from $B(X,Y,Z)\overunderslash{g=1}{c_1}X'_g$ $\overunderslash{i=1}{c_3}Y'_i\overunderslash{k=1}{c_5}Z'_k\boxtimes B(X,Y,Z)\overunderslash{h=1}{c_2}X''_h\overunderslash{j=1}{c_4}Y''_j$ $\overunderslash{l=1}{c_6}Z''_l$.
Because there are no other vertices in $B(X,Y,Z)\overunderslash{g=1}{c_1}X'_g$ $\overunderslash{i=1}{c_3}Y'_i\overunderslash{k=1}{c_5}Z'_k$ $\boxtimes B(X,Y,Z)\overunderslash{h=1}{c_2}X''_h$ $\overunderslash{j=1}{c_4}Y''_j\overunderslash{l=1}{c_6}Z''_l$, we have that $B(X,$ $Y,Z)\cong B(X,Y,Z)\overunderslash{g=1}{c_1}X'_g\overunderslash{i=1}{c_3}Y'_i$ $\overunderslash{k=1}{c_5}Z'_k \boxbackslash B(X,Y,Z)\overunderslash{h=1}{c_2}X''_h\overunderslash{j=1}{c_4}Y''_j$ $\overunderslash{l=1}{c_6}Z''_l$.
The proof for $[X,Y]$ and $[Y,Z]$ containing only backward arcs is similar.
This completes the proof of Lemma~\ref{lemma2}.
\end{proof}
Note, if $c_1=c_3=c_5=1$ then we have for $B(X,Y,Z)$ that each vertex $u_{g,h}$ of $X$ corresponds to a vertex $\tilde{x}''_{g,h}$ of $B(X,Y,Z)\overunderslash{h=1}{c_2}X''_h\overunderslash{j=1}{c_4}Y''_j\overunderslash{l=1}{c_6}Z''_l$, the vertex $v_{i,j}$ of $Y$ corresponds to the vertex $\tilde{y}''_{i,j}$ of $B(X,Y,Z)\overunderslash{h=1}{c_2}X''_h\overunderslash{j=1}{c_4}Y''_j\overunderslash{l=1}{c_6}Z''_l$, the vertex $w_{k,l}$ of $Z$ corresponds to the vertex $\tilde{z}''_{k,l}$ of $B(X,Y,Z)\overunderslash{h=1}{c_2}X''_h\overunderslash{j=1}{c_4}Y''_j\overunderslash{l=1}{c_6}Z''_l$, each arc $u_{g,h}v_{i,j}$ corresponds to an arc $\tilde{x}''_{g,h}\tilde{y}''_{i,j}$ of $B(X,Y,Z)\overunderslash{h=1}{c_2}X''_h$ $\overunderslash{j=1}{c_4}Y''_j\overunderslash{l=1}{c_6}Z''_l$ and each arc $v_{i,j}w_{k,l}$ corresponds to an arc $\tilde{y}''_{i,j}\tilde{z}''_{k,l}$ of $B(X,Y,Z)\overunderslash{h=1}{c_2}X''_h\overunderslash{j=1}{c_4}Y''_j\overunderslash{l=1}{c_6}Z''_l$.
This gives $B(X,Y,Z)\cong B(X,Y,Z)\overunderslash{h=1}{c_2}X''_h\overunderslash{j=1}{c_4}Y''_j\overunderslash{l=1}{c_6}Z''_l$.
Furthermore, we have for $B(X,Y,Z)$ that the vertices of $X$ correspond to the vertex $\tilde{x}'$ of $B(X,Y,Z)\overunderslash{i=1}{c_1}X'_i$ $\overunderslash{i=1}{c_3}Y'_i\overunderslash{i=1}{c_5}Z'_i$, the vertices of $Y$ of $B(X,Y,Z)\overunderslash{i=1}{c_1}X'_i$ $\overunderslash{i=1}{c_3}Y'_i\overunderslash{i=1}{c_5}Z'_i$ correspond to the vertex $\tilde{y}'$, the vertices of $Z$ correspond to the vertex $\tilde{z}'$ of $B(X,Y,Z)\overunderslash{i=1}{c_1}X'_i$ $\overunderslash{i=1}{c_3}Y'_i\overunderslash{i=1}{c_5}Z'_i$, each arc $u_{g,h}v_{i,j}$ corresponds to the arc $\tilde{x}'\tilde{y}'$ of $B(X,Y,Z)\overunderslash{h=1}{c_1}X'_h$ $\overunderslash{j=1}{c_3}Y'_j\overunderslash{l=1}{c_5}Z'_l$ and each arc $v_{i,j}w_{k,l}$ corresponds to the arc $\tilde{y}'\tilde{z}'$ of $B(X,Y,Z)$ $\overunderslash{h=1}{c_1}X'_h\overunderslash{j=1}{c_3}Y'_j\overunderslash{l=1}{c_5}Z'_l$.
Then, $B(X,Y,Z)\overunderslash{i=1}{c_1}X'_i\overunderslash{i=1}{c_3}Y'_i\overunderslash{i=1}{c_5}Z'_i$ is a path from $\tilde{x}'$ to $\tilde{z}'$ and we have that  $B(X,Y,Z)\cong B(X,Y,Z)\overunderslash{i=1}{c_1}X'_i\overunderslash{i=1}{c_3}Y'_i\overunderslash{i=1}{c_5}Z'_i\boxbackslash B(X,Y,Z)\overunderslash{i=1}{c_2}X''_i\overunderslash{i=1}{c_4}Y''_i\overunderslash{i=1}{c_6}Z''_i$.
Because there is no reduction of the number of vertices (and arcs) in $B(X,Y,Z)\overunderslash{h=1}{c_2}X''_h\overunderslash{j=1}{c_4}Y''_j\overunderslash{l=1}{c_6}Z''_l$ with respect to $B(X,Y,Z)$, this is a useless decomposition.
Likewise, for $c_2=c_4=c_6=1$, we have as well such a useless decomposition.
Therefore, at least one of the values of $GCD(c_1,c_2)$, $GCD(c_3,c_4)$ or $GCD(c_5,c_6)$ has to be greater than one, or, in case $|X|,|Y|$ and $|Z|$ are prime numbers, at least one but not all of the $c_1,c_3,c_5$ have to be greater than one (and, therefore, at least one but not all of the $c_2,c_4,c_6$ is greater than one).

We continue with the decomposition of an $n$-partite graph $B(X_1,\ldots,X_n)$ where all arcs have the same label, the arcs in $[X_1,X_2],\ldots,[X_{n-1},X_{n}]$ are the only arcs of $B(X_1,\ldots,X_n)$, the subgraph $B(X_1,X_2)$ of $B(X_1,\ldots,X_n)$ induced by $X_1\cup X_2$ is a complete bipartite graph, each subgraph $B(X_i,\chi_{i+1}), i=2,\ldots,n-1,$ of $B(X_1,\ldots,X_n)$ arc induced by the arcs of $[X_i,X_{i+1}]$ is a complete bipartite graph with $\chi_{i+1}\subseteq X_{i+1}$ (note that $\chi_{n}=X_{n}$). 
The partition of $X_1$ in subsets of $X_1$ is similar to the partition of $X$ in subsets of $X$ in Lemma~\ref{lemma2}, the partition of each of the $X_{2},\ldots,X_{n-1}$ in subsets of $X_{2},\ldots,X_{n-1}$ is similar to the partition of $Y$ in subsets of $Y$ in Lemma~\ref{lemma2},  the partition of each of the $\chi_{2},\ldots,\chi_{n-1}$ in subsets of $\chi_{2},\ldots,\chi_{n-1}$ is similar to the partition of $W$ in subsets of $W$ in Lemma~\ref{lemma2} and the partition of $X_n$ in subsets of $X_n$ is similar to the partition of $Z$ in subsets of $Z$ in Lemma~\ref{lemma2}.
Following these requirements, we state and prove the following decomposition theorem.

\begin{theorem}\label{theorem}
Let $B(X_1,\ldots, X_n)$ be a weakly connected $n$-partite graph where the labels of all arcs are the same.
Let $[X_1,X_{2}],\ldots,[X_{n-1},X_n]$ contain the only arcs of $B(X_1,\ldots, X_n)$.
Let $[X_i,X_{i+1}]$ contain only forward arcs for all $i\in\{1,\ldots,n-1\}$ or let $[X_i,X_{i+1}]$ contain only backward arcs for all $i\in\{1,\ldots,n-1\}$.
Let $c_{1,1},c_{1,2},c_{2,3},\ldots,c_{n-1,3},c_{2,4},\ldots,c_{n-1,4},c_{n,5},c_{n,6},c_{2,7},\ldots,$ $c_{n-1,7}\in~\hspace{-5pt}\mathbb{N}$.
Let $|X_1|=c_{1,1}\cdot c_{1,2}$.
Let $B(X_1,X_2)$ be the complete bipartite subgraph of $B(X_1,\ldots,$ $X_n)$ induced by $X_1\cup X_2$.
Let $B(\chi_{i},X_{i+1}), \chi_{i}\subseteq X_{i}, i\in\{2,\ldots,n-1\} $ be the complete bipartite subgraph of $B(X_1,\ldots,X_n)$ arc-induced by the arcs in $[X_{i},X_{i+1}]$.
Let $GCD(|X_{i}|,|\chi_{i}|)=c_{i,3}$, $|X_{i}|=c_{i,3}\cdot c_{i,4}, |\chi_{i}|=c_{i,3}\cdot c_{i,7}$, $i\in\{2,\ldots,n-1\}$.
Let $|X_{n}|=c_{n,5}\cdot c_{n,6}$.
Then there exist $X'_{1,g}, X'_{m,i},X'_{n,k},X''_{1,h},$ $X''_{m,j},X''_{n,l}$ such that $B(X_1,\ldots,X_n)\cong B(X_1,\ldots,X_n)
\overunderslash{g=1}{c_{1,1}}X'_{1,g}
\overunderslash{m=2}{n-1}\,
\overunderslash{i=1}{c_{i,3}}X'_{m,i}$ $
\overunderslash{k=1}{c_{n,5}}X'_{n,k}\boxbackslash B(X_1,\ldots,$ $X_n)
\overunderslash{h=1}{c_{1,2}}X''_{1,h}
\overunderslash{m=2}{n-1}\,
\overunderslash{j=1}{c_{j,4}}X''_{m,j}
\overunderslash{l=1}{c_{n,6}}X''_{n,l}$.
\end{theorem}
\begin{proof}
%Proof by induction. For $n=2$, we apply Lemma~\ref{lemma1}. For $n=3$, we apply Lemma~\ref{lemma2} and we have the induction step from $B(X_1,X_2)$ to $B(X_1,X_2,X_3)$.
Proof by induction. For $n=2$, we apply Lemma~\ref{lemma1}. For $n=3$, we apply Lemma~\ref{lemma2}.
Let $B(X_1,\ldots,X_{n-1})$ be decomposed into two $(n-1)$-partite graphs $B(X_1,\ldots,$ $X_{n-1})
\overunderslash{g=1}{c_{1,1}}X'_{1,g}
\overunderslash{m=2}{n-2}\,
\overunderslash{i=1}{c_{i,3}}$ $X'_{m,i}
\overunderslash{k=1}{c_{n-1,5}}X'_{n-1,k}$ and $B(X_1,\ldots,X_{n-1})
\overunderslash{h=1}{c_{1,2}}X''_{1,h}\,
\overunderslash{m=2}{n-2}\,
\overunderslash{j=1}{c_{j,4}}X''_{m,j}$ $
\overunderslash{l=1}{c_6}X''_{n-1,l}$,
such that $B(X_1,\ldots,X_{n-1})$ $\cong B(X_1,X_2,\ldots,$ $X_{n-1})
\overunderslash{g=1}{c_{1,1}}X'_{1,g}\,
\overunderslash{m=2}{n-2}\,
\overunderslash{i=1}{c_{i,3}}X'_{m,i}
\overunderslash{k=1}{c_{n-1,5}}X'_{n-1,k}\,\boxbackslash\, B(X_1,\ldots,X_{n-1})
\overunderslash{h=1}{c_{1,2}}X''_{1,h}\,
\overunderslash{m=2}{n-2}\,
\overunderslash{j=1}{c_{i,4}}X''_{m,j}$ $
\overunderslash{l=1}{c_{n-1,6}}X''_{n-1,l}$. 
Then, for $B(X_1,\ldots,X_n)$, we have the partition of $X_{n-1}$ in the sets $X'_{n-1,i},i=1,\ldots c_{n-1,3}$ and the partition of $X_{n-1}$ in the sets $X''_{n-1,j},j=1,\ldots c_{n-1,4}$, the partition of $\chi_{n-1}\subseteq X_{n-1}$ in the sets $\chi'_{n-1,i},i=1,\ldots c_{n-1,3}$ and the partition of $\chi_{n-1}$ in the sets $\chi''_{n-1,j},j=1,\ldots c_{n-1,7}$, and the partition of $X_n$ in the sets $X'_{n,k},k=1,\ldots c_{n,5}$ and the partition of $X_n$ in the sets $X''_{n,l},l=1,\ldots c_{n,6}$.
Now, with similar arguments as for $Y$ and $Z$ of Lemma~\ref{lemma2}, we have that $B(X_1,\ldots,X_n)$ $\cong B(X_1,\ldots,X_n)$ $
\overunderslash{g=1}{c_{1,1}}X'_{1,g}
\overunderslash{m=2}{n-1}\,
\overunderslash{i=1}{c_{i,3}}$ $X'_{m,i}
\overunderslash{k=1}{c_{n,6}}X'_{n,k}$ $\boxbackslash $ $B(X,\ldots,X_n)
\overunderslash{h=1}{c_{1,2}}X''_{1,h}
\overunderslash{m=2}{n-1}\,%$ $
{\overunderslash{j=1}{c_{j,4}}X''_{m,j}
\overunderslash{l=1}{c_{n,7}}X''_{n,l}}$.
%The proof for $[X_i,X_{i+1}], i=1,\ldots,n-1$ contain only backward arcs is similar.
This completes the proof of Theorem~\ref{theorem}.
\end{proof}

%\section{Future work}

%\section*{Acknowledgement}
%The authors would like to express their gratitude to the anonymous reviewers for their useful suggestions and comments. 

\end{document}